\documentclass{amsart}
\usepackage{amsmath, amsthm, amscd, amssymb,latexsym,enumerate}
\usepackage{url}
\usepackage{array}
\usepackage[all]{xy}
\usepackage{enumitem}
\usepackage{hyperref}

\def\Z{\mathbb{Z}}
\def\Q{\mathbb{Q}}

\def\R{\mathbb{R}}
\def\C{\mathbb{C}}

\def\r{{r}}

\def\Q{{\mathbb Q}}

\def\G{{G}}
\def\Z{{\mathbb Z}}

\def\C{{\mathbb C}}
\def\R{{\mathbb R}}

\def\WPic{\mathrm{WPic}}
\def\Pic{\mathrm{Pic}}

\def\Pic{\mathrm{Pic}}

\def\Aut{\mathrm{Aut}}

\def\M{\mathrm{M}}

\def\Hom{\mathrm{Hom}}

\def\GL{\mathrm{GL}}

\def\dim{\mathrm{dim}}

\def\log{\mathrm{log}}
\def\det{\mathrm{det}}
\def\rk{\mathrm{rank}}
\def\rank{\mathrm{rank}}
\def\I{{\mathcal I}}

\def\m{{\mathfrak m}}

\def\FFF{{\mathcal F}}

\def\isom{\xrightarrow{\sim}}

\numberwithin{equation}{section}

\newtheorem{thm}{Theorem}
\numberwithin{thm}{section}

\newtheorem{lem}[thm]{Lemma}
\newtheorem{cor}[thm]{Corollary}
\newtheorem{prop}[thm]{Proposition}

\theoremstyle{definition}
\newtheorem{defn}[thm]{Definition}
\newtheorem{notation}[thm]{Notation}
\newtheorem{algorithm}[thm]{Algorithm}

\newtheorem{ex}[thm]{Example}
\newtheorem{exs}[thm]{Examples}
\newtheorem{rem}[thm]{Remark}

\title[Lattices with Symmetry]{Lattices with Symmetry}
\author[H.\ W.\ Lenstra, Jr.]{H.\ W.\ Lenstra, Jr.}
\address{Mathematisch Instituut, Universiteit Leiden, The Netherlands}
\email{hwl@math.leidenuniv.nl}
\author[A.\ Silverberg]{A.\ Silverberg}
\address{Department of Mathematics, University of California, Irvine, CA 92697, USA}
\email{asilverb@uci.edu}

\begin{document}

\thanks{This material is based on research sponsored by DARPA under agreement numbers FA8750-11-1-0248 and FA8750-13-2-0054 and by the Alfred P.~Sloan Foundation. The U.S.\ Government is authorized to reproduce and distribute reprints for Governmental purposes notwithstanding any copyright notation thereon. The views and conclusions contained herein are those of the authors and should not be interpreted as necessarily representing the official policies or endorsements, either expressed or implied, of DARPA or the U.S.\ Government.
}

\begin{abstract} 
For large ranks, there is no good algorithm that decides whether a 
given lattice has an orthonormal basis. But when the lattice is 
given with enough symmetry, we can construct a provably deterministic
polynomial-time algorithm 
to accomplish this, based on the work of Gentry and Szydlo.
The techniques involve algorithmic algebraic number theory,
analytic number theory, commutative algebra, and lattice basis reduction.
\end{abstract}

\maketitle

\section{Introduction}

Let $G$ be a finite abelian group and let $u\in G$  be a fixed element of order $2$.
Define a $G$-lattice to be an integral lattice $L$ with an 
action of $G$ on $L$ that preserves the inner product, such that
$u$ acts as $-1$.
The {\em standard} $G$-lattice is the modified group ring 
$\Z\langle\G\rangle = \Z[G]/(u+1)$,
equipped with a natural inner product; we refer to Sections
\ref{backgroundsect}, \ref{backgroundsect2}, and \ref{modgpringsect} for more precise definitions.
Our main result reads as follows:

\begin{thm}
\label{mainthm}
There is a deterministic polynomial-time algorithm 
that, given a finite abelian group $\G$ with 
an element $u$ of order $2$, and a $\G$-lattice $L$,
decides whether $L$ and $\Z\langle\G\rangle$ are isomorphic as $G$-lattices, 
and if they are, exhibits such an isomorphism.
\end{thm}

We call a $G$-lattice $L$ {\em invertible} if it is unimodular and
there is a  $\Z\langle G\rangle$-module $M$ such that
$L \otimes_{\Z\langle G\rangle} M$ and $\Z\langle G\rangle$
are isomorphic as $\Z\langle G\rangle$-modules
(see Definition \ref{invertdef2} and Theorem \ref{invertequiv}).
For example, the standard $G$-lattice is invertible.
The following result is a consequence of 
Theorem \ref{mainthm}.

\begin{thm}
\label{mainthmcor}
There is a deterministic polynomial-time algorithm 
that, given a finite abelian group $\G$ equipped with 
an element of order $2$, and invertible $\G$-lattices $L$ and $M$,
decides whether $L$ and $M$ are isomorphic as $G$-lattices, 
and if they are, exhibits such an isomorphism.
\end{thm}

Deciding whether two lattices are isomorphic is a
notorious problem. Our results show that it admits a satisfactory
solution if the lattices are equipped with 
sufficient structure.

Our algorithms and runtime estimates draw upon an array of techniques
from algorithmic algebraic number theory,
commutative algebra,
lattice basis reduction,
and analytic number theory.
We develop techniques from commutative algebra that have not yet
been fully exploited in the context of cryptology.

An important ingredient to our algorithm is a powerful
novel technique that was invented by C.~Gentry and M.~Szydlo
in Section 7 of \cite{GS}. 
We recast their method in the language of commutative algebra,
replacing the ``polynomial chains'' that they used to compute
powers of ideals in certain rings by tensor powers of modules.
A number of additional changes enabled us to obtain a
{\em deterministic} polynomial-time algorithm, whereas the
Gentry-Szydlo algorithm is at best probabilistic.

The technique of Gentry and Szydlo has seen several applications in
cryptography, as enumerated in \cite{LenSil}.
By placing it in an algebraic framework, we have already
been able to generalize the method significantly, replacing
the rings $\Z[X]/(X^n-1)$ (with $n$ an odd prime) used by
Gentry and Szydlo by the larger class of modified group rings
that we defined above, and further extensions appear to be possible.
In addition, we hope that our reformulation will make it easier to
understand the method and improve upon it.
This should help to make it more widely applicable in a cryptographic
context.

\subsection{Overview of algorithm proving Theorem \ref{mainthm}}
The algorithm starts by testing whether the given $G$-lattice
$L$ is {\em invertible}, which is a necessary condition for being
isomorphic to the standard $G$-lattice.
Invertibility is a concept with several attractive properties.
For example, it is easy to test. 
Secondly, every invertible $G$-lattice has rank $\#G/2$ and 
determinant $1$, and therefore can be specified using a small number of bits
(Proposition~\ref{LLLlem} below, and the way it is used to
prove Theorem~\ref{WPgpfin}).
Thirdly, an invertible $G$-lattice
$L$ is isomorphic to the standard one if and only if there is a
{\em short} element $e\in L$, that is, an element of length $1$.

Accordingly, most of the algorithm consists of looking for
short elements in invertible $G$-lattices, or proving that none exists.
The main tool for this is a further property of invertible $G$-lattices,
which concerns {\em multiplication}. 
As the name suggests, any invertible $G$-lattice $L$ has an {\em inverse}
$\overline{L}$, which is also an invertible $G$-lattice, and any two
invertible $G$-lattices $L$ and $M$ can be {\em multiplied} using
a tensor product operation, which yields again an invertible $G$-lattice.
For example, the product of $L$ and $\overline{L}$ is the standard
$G$-lattice $\Z\langle G\rangle$.

No sequence of multiplications will ever give rise to
coefficient blow-up since, as remarked above, 
every invertible $G$-lattice can be specified using a small number of bits.
It suffices to take the simple precaution
of performing a lattice basis reduction after every multiplication
(as in Algorithm \ref{LMmultalg}). 
It is a striking consequence that 
even very high powers $L^r$ of $L$ can be efficiently computed!

Each short element $e\in L$ gives rise to a short element
$e^r\in L^r$, which may be thought of as the $r$-th power of $e$.
If $r$ is well-chosen ($r=k(\ell)$, in the notation of Algorithm \ref{algor}),
then $e^r$ will satisfy a congruence condition (modulo $\ell$),
and if we take $\ell$ large enough
this enables us to determine $e^r$ (or show that no $e$ exists).
However, passing directly from $e^r$ to $e$ is infeasible due to the
large size of $r$. Thus, one also finds $e^s\in L^s$ for a 
second well-chosen large number $s$ ($=k(m)$, in Algorithm \ref{algor}),
and a multiplicative combination of $e^r$ and $e^s$ yields
$e^{\gcd(r,s)}\in L^{\gcd(r,s)}$.
A result from analytic number theory shows that $r$ and $s$ can be
chosen such that $\gcd(r,s)$ ($=k$, in Algorithm \ref{algor})
is so small that $e$, if it exists, can be found from
$e^{\gcd(r,s)}$ by a relatively easy root extraction.
The latter step requires techniques (Proposition \ref{rootsofunitylem}) 
of a nature entirely different from those in the present paper,
and is therefore delegated to a separate publication \cite{rootsofunity}.

While we believe that the techniques introduced here could lead
to practical algorithms, we did not attempt an actual
implementation. Also, any choices and recommendations we
made were inspired by the desire to give a clean proof of our
theorem rather than efficient algorithms.

\subsection{Structure of the paper}
Sections \ref{backgroundsect}--\ref{latticecosetssect} contain background on 
integral lattices. In particular, we derive a new bound 
for the entries of a matrix describing
an automorphism of a unimodular lattice with respect to a reduced basis 
(Proposition \ref{LLLlem}).
Sections \ref{backgroundsect2}--\ref{recovergssect}
contain basic material about $G$-lattices and modified group rings.
Important examples of $G$-lattices are the ideal lattices introduced
in Section \ref{ideallatsect}.
In Remark \ref{GSrmk} we explain how to recover the Gentry-Szydlo algorithm from 
Theorem \ref{mainthmcor}.
In Sections 
\ref{conceptsforpfsect}--\ref{equivinvertdefsect} we begin our
study of 
invertible $G$-lattices,  giving several equivalent definitions and
an algorithm for recognizing invertibility.
Section \ref{shorttensorsect} is devoted to the following pleasing
result: a $G$-lattice is $G$-isomorphic to the
standard one
if and only if it is invertible and has a vector
of length $1$.
In Sections \ref{tensoringsect}--\ref{WPsect} we show
how to multiply invertible $G$-lattices and we introduce
the Witt-Picard group of $\Z\langle G\rangle$, of which the elements
correspond to 
$G$-isomorphism classes of invertible $G$-lattices.
It has properties reminiscent of the class group in algebraic number
theory; in particular, it is a finite abelian group 
(Theorems \ref{WPicabgpthm} and \ref{WPgpfin}).
We also show how to do computations in the Witt-Picard group.
In Section \ref{tensorsect} we treat the extended tensor algebra $\Lambda$,
which is in a sense the hero of story: it is a single algebraic structure
that comprises all rings and lattices occurring in our main algorithm.
Section \ref{shortvectalgsect} shows how $\Lambda$ can be used to assist
in finding vectors of length $1$. In Section \ref{LinnikApp} we 
use Linnik's theorem from analytic number theory in order to find auxiliary
numbers in our main algorithm, and our main algorithm is presented 
in Section \ref{algorsect}. 

\subsection{Notation}
For the purposes of this paper, commutative rings have an identity
element $1$, which may be $0$. If $R$ is a commutative ring,
let $R^\ast$ denote the group of elements of $R$ that
have a multiplicative inverse in $R$.

\section{Integral lattices}
\label{backgroundsect}

We begin with some background on lattices and on
lattice automorphisms (see also \cite{HWLMSRI}).

\begin{defn}
A {\bf lattice} or {\bf integral lattice} is a finitely generated abelian group $L$ 
 with a map $\langle \, \cdot \, , \,  \cdot \,  \rangle : L \times L \to \Z$ that is
\begin{itemize}

\item bilinear: 
$\langle x, y+z \rangle = \langle x,y \rangle + \langle x,z \rangle$ and
$\langle x + y, z \rangle = \langle x,z \rangle + \langle y,z \rangle$
for all $x,y,z\in L$,
\item symmetric: $\langle x, y \rangle = \langle y, x \rangle$ for all $x,y\in L$, and
\item positive definite: 
$\langle x, x \rangle > 0$ if $0\neq x\in L$.
\end{itemize}
\end{defn}

As a group, $L$ is isomorphic to $\Z^n$ for some $n \in \Z_{\ge 0}$, 
which is called the {\bf rank} of $L$ and is denoted $\rank(L)$.
In algorithms, 
a lattice is specified by a 
Gram matrix $(\langle b_i,b_j\rangle )_{i,j=1}^n$ associated to a
$\Z$-basis $\{ b_1,\ldots, b_n \}$ and an element of a lattice
is specified by its coefficient vector on the same basis.
The inner product
$\langle \, \cdot \, , \,  \cdot \,  \rangle$ 
extends to a real-valued inner product on $L\otimes_\Z \R$
and makes $L\otimes_\Z \R$ into a Euclidean vector space.

\begin{defn}
The {\bf standard lattice} of rank $n$ is $\Z^n$ with
$\langle x, y \rangle = \sum_{i=1}^n x_iy_i.$
Its Gram matrix is the $n \times n$ identity matrix. 
\end{defn}

\begin{defn}
\label{unimodulardef}
The determinant $\det(L)$ of a lattice $L$ is the determinant of the Gram matrix of $L$;
equivalently, $\det(L)$ is the order of the cokernel of the map $L \to \Hom(L,\Z)$,
$x \mapsto (y \mapsto \langle x,y \rangle)$.
A lattice $L$ is {\bf unimodular} if this map 
is bijective, i.e., if $\det(L)=1$.
\end{defn}

\begin{defn}
\label{latisomdefn}
An {\bf isomorphism} $L \isom M$ of lattices is a group
isomorphism $\varphi$ from $L$ to $M$ that respects the lattice
structures, i.e.,
$$
\langle \varphi(x), \varphi(y) \rangle = \langle x, y \rangle
$$
for all $x,y\in L$.
If such a map $\varphi$ exists, then $L$ and $M$ are {\bf isomorphic} lattices.
An {\bf automorphism} of a lattice $L$ is an isomorphism from $L$ to itself.
The set of automorphisms of $L$ is a finite group $\Aut(L)$ whose center
contains $-1$. 
\end{defn}

In algorithms, 
isomorphisms are specified by their matrices on the given bases of $L$ and $M$.

\begin{exs} $ $
\begin{enumerate}
\item
``Random'' lattices have $\Aut(L) = \{\pm 1\}$.
\item
Letting $S_n$
denote the symmetric group on $n$ letters and $\rtimes$ denote
 semidirect product, we have
$\Aut(\Z^n) \cong \{\pm 1\}^n \rtimes S_n$. (The standard basis vectors can
be permuted, and signs changed.)
\item
If $L$ is the equilateral triangular lattice in the plane,  
then $\Aut(L)$ is the symmetry group of the regular hexagon, which is
a dihedral group of order $12$.
\end{enumerate}
\end{exs}

\section{Reduced bases and automorphisms}
\label{LLLmiscsect}

The main result of this section is Proposition \ref{LLLlem},
in which we obtain some bounds for LLL-reduced
bases of unimodular lattices.
We will use this result to give bounds on the complexity of our algorithms  
and to show that
the Witt-Picard group (Definition \ref{WPgpdefn} below) is finite.
If $L$ is a lattice and $a \in L\otimes_\Z \R$, let $|a| = \langle a,a \rangle^{1/2}$.

\begin{defn}
\label{LLLreduceddefn}
If $\{ b_1,\ldots,b_n\}$ is a basis for a lattice $L$,
and $\{ b_1^\ast,\ldots,b_n^\ast\}$
is its Gram-Schmidt orthogonalization, 
and $$b_i = b_i^\ast + \sum_{j=1}^{i-1} \mu_{ij}b_j^\ast$$
with $\mu_{ij}\in \R$,
then
$\{ b_1,\ldots,b_n\}$ is {\bf LLL-reduced} if 
\begin{enumerate}
\item
$|\mu_{ij}| \le \frac{1}{2}$ for all $j<i\le n$, and 
\item
$|b_i^\ast|^2 \le 2|b_{i+1}^\ast|^2$
for all $i<n$.
\end{enumerate}
\end{defn}

\begin{rem}
\label{LLLredrmk}
The LLL basis reduction algorithm \cite{LLL} takes as input a lattice,
and produces an LLL-reduced basis of the lattice, in polynomial time.
\end{rem}

\begin{lem}
\label{matrixentrybd}
If $a = (\mu_{ij})_{ij}\in \M(n,\R)$ is a lower-triangular real matrix with
$\mu_{ii} = 1$ for all $i$ and $|\mu_{ij}| \le 1/2$ for all $j < i$,
and $a^{-1} = (\nu_{ij})_{ij}$,
then 
$$|\nu_{ij}| \le 
\begin{cases}
0 & \text{ if $i< j$} \\
1 & \text{ if $i= j$} \\
\frac{1}{3}\left(\frac{3}{2}\right)^{i-j} & \text{ if $i> j$.} 
\end{cases}
$$
\end{lem}

\begin{proof}
Define $e\in \M(n,\R)$  by $e_{ij} = 0$
if $j\ge i$ and $e_{ij} = \frac{1}{2}$
if $j < i$.
Define $h\in \M(n,\R)$ by  
$h_{i+1,i} = 1$ for $i=1,\ldots,n-1$ and $h_{ij}  = 0$ otherwise.
Then $$e = \sum_{j=1}^\infty \frac{1}{2}h^j = \frac{h}{2(1-h)}.$$
Thus, 
$1-e = (1-3h/2)/(1-h)$ and
\begin{multline*}
(1-e)^{-1} = (1-h)/(1-3h/2) \\
= (1-h)\sum_{j=0}^\infty \left(\frac{3}{2}\right)^jh^j  
 =
\sum_{j=0}^\infty \left(\frac{3}{2}\right)^jh^j - 
\sum_{j=0}^\infty \left(\frac{3}{2}\right)^jh^{j+1} = \\
\begin{pmatrix}
1 & 0 & \cdots & 0 \\
\frac{3}{2} & 1 & \cdots & 0 \\
\phantom{2}(\frac{3}{2})^2 & \frac{3}{2} & \cdots & 0 \\
\vdots & \vdots & \ddots & \vdots \\
\phantom{n-}(\frac{3}{2})^{n-1} & \phantom{n-}(\frac{3}{2})^{n-2} & \cdots & 1
\end{pmatrix} - 
\begin{pmatrix}
0 & 0 & \cdots & 0 & 0 & 0 \\
1 & 0 & \cdots & 0 & 0 & 0 \\
\frac{3}{2} & 1 & \cdots  & 0 & 0 & 0 \\
\vdots & \vdots & \ddots & \vdots  & \vdots & \vdots \\
\phantom{n-}(\frac{3}{2})^{n-2} & \phantom{n-}(\frac{3}{2})^{n-3} & \cdots & \frac{3}{2} & 1 & 0
\end{pmatrix},
\end{multline*}
which has $ij$ entry
$0$ if $i< j$, and 
$1$ if $i= j$, and
$\frac{1}{3}\left(\frac{3}{2}\right)^{i-j}$ if $i> j$.

Since $e^n = 0 = (1-a)^n$, we have
$$(1-e)^{-1} = \sum_{i=0}^{n-1} e^i \quad \text{ and } \quad
a^{-1} = \sum_{i=0}^{n-1} (1-a)^i.$$
If $c = (c_{ij})_{ij}\in \M(n,\R)$, let $|c|$ denote $(|c_{ij}|)_{ij}$.
If $c,d\in \M(n,\R)$, then 
$c \le d$ means that $c_{ij} \le d_{ij}$ for all $i$ and $j$.
We have $$|a^{-1}| \le \sum_{i=0}^{n-1} |1-a|^i \le \sum_{i=0}^{n-1} e^i = (1-e)^{-1}.$$
This gives the desired result.
\end{proof}

\begin{prop}
\label{LLLlem}
If $\{ b_1,\ldots,b_n\}$ is an LLL-reduced basis for an
integral unimodular lattice $L$ and $\{ b_1^\ast,\ldots,b_n^\ast\}$
is its Gram-Schmidt orthogonalization, then
\begin{enumerate}
\item
$2^{1-i} \le |b_i^\ast|^2 \le 2^{n-i}$,
\item
$|b_i|^2 \le 2^{n-1}$  for all $i \in \{ 1,\ldots,n\}$,
\item
$|\langle b_i,b_j\rangle| \le 2^{n-1}$ for all $i$ and $j$,
\item
if $\sigma\in\Aut(L)$,
and for each $i$ we have 
$\sigma(b_i) = \sum_{j=1}^n a_{ij}b_j$ with $a_{ij} \in \Z$,
then $|a_{ij}| \le 3^{n-1}$ for all $i$ and $j$.
\end{enumerate}
\end{prop}

\begin{proof}
It follows from Definition \ref{LLLreduceddefn} that for all 
$1 \le j \le i \le n$ we have
$|b_i^\ast|^2 \le 2^{j-i}|b_j^\ast|^2$, so
for all $i$ we have 
$$
2^{1-i}|b_1^\ast|^2 \le |b_i^\ast|^2 \le 2^{n-i}|b_n^\ast|^2.
$$

Since $L$ is integral we have 
$$
|b_1^\ast|^2 = |b_1|^2 = \langle b_1,b_1\rangle \ge 1,$$ 
so $|b_i^\ast|^2 \ge 2^{1-i}$.
Letting $L_i = \sum_{j=1}^i \Z b_j$, we have 
$$
|b_i^\ast| = \det(L_i)/\det(L_{i-1}).
$$
Since $L$ is integral and unimodular, we have
$$|b_n^\ast| = \det(L_n)/\det(L_{n-1}) = 1/\det(L_{n-1}) \le 1,$$
so $|b_i^\ast|^2 \le 2^{n-i}$, giving (i).

Since $\{ b_i^\ast \}$ is orthogonal we have
\begin{multline*}
|b_i|^2 = |b_i^\ast|^2 + \sum_{j=1}^{i-1} \mu_{ij}^2|b_j^\ast|^2
\le 2^{n-i} + \frac{1}{4}\sum_{j=1}^{i-1} 2^{n-j} \\ =
2^{n-i} + (2^{n-2} - 2^{n-i-1}) = 2^{n-2} + 2^{n-i-1} \le 2^{n-1},
\end{multline*}
giving (ii). Now (iii) follows by applying the
Cauchy-Schwarz inequality $|\langle b_i,b_j\rangle| \le |b_i||b_j|$
and (ii).

For (iv), define $\{ c_1,\ldots,c_n\}$ to be the basis of $L$ that is dual to
$\{ b_1,\ldots,b_n\}$, i.e., $\langle c_i,b_j\rangle = \delta_{ij}$
for all $i$ and $j$, where $\delta_{ij}$ is the Kronecker delta symbol.
Then $a_{ij} = \langle c_j,\sigma(b_i)\rangle$ so
\begin{equation}
\label{aijfirstbd}
|a_{ij}| \le |c_j||\sigma(b_i)| = |c_j||b_i|. 
\end{equation}

Define $\mu_{ii} = 1$ for all $i$ and $\mu_{ij}=0$ if $i<j$,
and let 
$$
M = (\mu_{ij})_{ij} \in \M(n,\R).
$$
Then $$(b_1 \, \, b_2 \, \cdots  \, b_n) = (b_1^\ast \, \, b_2^\ast \, \cdots \, b_n^\ast)M^t.$$ 
For $0\neq x\in L \otimes_\Z\R$, define $$x^{-1} = \frac{x}{\langle x,x\rangle}.$$
This inverse map is characterized by the properties that $\langle x,x^{-1}\rangle = 1$ and $\R x^{-1} = \R x$;
so $(x^{-1})^{-1} =x$.
Since the basis dual to $\{ b_i^\ast \}_i$ is $\{ (b_i^\ast)^{-1} \}_i$, and
$M$ gives the change of basis from $\{ b_i^\ast \}_i$ to $\{ b_i \}_i$, 
it follows that the matrix
$(M^t)^{-1}$ gives the change of basis from  $\{ (b_i^\ast)^{-1} \}_i$ to
$\{ c_i \}_i$. Thus,
$$
(c_1  \,  \, \cdots \,  \,  c_n) = 
((b_1^\ast)^{-1} \,  \, \cdots \,  \, (b_n^\ast)^{-1})M^{-1}.
$$

Letting $(\nu_{ij})_{ij} = M^{-1}$,  
by Lemma \ref{matrixentrybd} we have
$$
c_j = \sum_{i\ge j}(b_i^\ast)^{-1}\nu_{ij}
$$
with $\nu_{ii}=1$ and $|\nu_{ij}| \le 
\frac{1}{3}\left(\frac{3}{2}\right)^{i-j}$ if $i> j$.
By (i) we have 
$$
|(b_i^\ast)^{-1}|^2 \le 2^{i-1}.
$$
Thus,
\begin{align*}
|c_j|^2 & \le \sum_{i\ge j}2^{i-1}\nu_{ij}^2  \\
& \le 
2^{j-1} + \frac{1}{9}\sum_{i> j}2^{i-1}\left(\frac{9}{4}\right)^{i-j} \\
&  \le 
2^{j-1} + \frac{2^{j-1}}{9}\sum_{k=1}^{n-j}\left(\frac{9}{2}\right)^{k} \\ 
& = 
2^{j-1} + \frac{2^{j}}{63}\left[\left(\frac{9}{2}\right)^{n-j+1}-\frac{9}{2}\right] \\
& = 
\frac{2^{j-1}}{7}\left[\left(\frac{9}{2}\right)^{n-j} + 6\right] \\ 
& \le
\frac{1}{7}\left(\frac{9}{2}\right)^{n-1} + \frac{6}{7}\left(\frac{9}{2}\right)^{n-1} 
 =
\left(\frac{9}{2}\right)^{n-1}.
\end{align*}
Now by (ii) and \eqref{aijfirstbd} we have $|a_{ij}|^2 \le 9^{n-1}$, as desired.
\end{proof}

\begin{rem}
It is easier to get the weaker bound $|a_{ij}| \le 2^{n \choose {2}}$, as follows.
Write $b_j = b_j^\# + y$ with $y \in \sum_{i\neq j} \R b_i$
and $b_j^\#$ orthogonal to $\sum_{i\neq j} \R b_i$.
With $c_j$ as in the proof of Proposition \ref{LLLlem},
we have $c_j = (b_j^\#)^{-1}$, by the characterizations of $(b_j^\#)^{-1}$ and $c_j$.
Since 
$$
1 = \det(L) = \det(\sum_{i\neq j} \Z b_i)|b_j^\#|
$$
we have 
$$
|c_j| = |\det(\sum_{i\neq j} \Z b_i)| \le
\prod_{i \neq j}|b_i| \le 2^{(n-1)^2/2}
$$
by Hadamard's inequality and Proposition \ref{LLLlem}(ii).
By \eqref{aijfirstbd} and Proposition \ref{LLLlem}(ii) we have
$|a_{ij}| \le 2^{n \choose {2}}$. 
\end{rem}

\section{Short vectors in lattice cosets}
\label{latticecosetssect}
We show how to find the unique vector of length $1$ in a
suitable lattice coset, when such a vector exists.

\begin{prop}
\label{findingekmprop}
Suppose $L$ is an integral lattice, $3 \le m\in\Z$,
and $C\in L/mL$. Then the coset $C$ contains at most one element $x\in L$ with 
$\langle x,x\rangle=1$.
\end{prop}

\begin{proof}
Suppose $x, y\in C$, with
$\langle x,x\rangle=\langle y,y\rangle=1$. 
Since $x,y\in C$, there exists
$w\in L$ such that $x-y=mw$.
Using the triangle inequality, we have
$$
m\langle w,w \rangle^{1/2} = \langle x-y, x-y \rangle^{1/2}
\le \langle x,x \rangle^{1/2} + \langle y,y \rangle^{1/2} 
= 1 + 1 = 2.
$$
Since $m \ge 3$ and $\langle w,w \rangle \in\Z_{\ge 0}$, we have $w=0$, and thus $y=x$.
\end{proof}

\begin{algorithm}
\label{findingekmalg2}
Given a rank $n$ integral lattice $L$, an integer $m$ such that
$m\ \ge 2^{n/2} +1$, and $C\in L/mL$, the algorithm computes
all $y\in C$ with $\langle y,y\rangle=1$.

\begin{enumerate}
\item
Compute an LLL-reduced basis for $mL$ and use it as in
\S 10 of \cite{HWLMSRI} to 
compute $y\in C$ such that 
$
\langle y,y \rangle \le (2^n -1)\langle x,x \rangle
$
for all $x\in C$, i.e., to find an approximate solution to the nearest vector problem.
\item
Compute $\langle y,y \rangle$. 
\item
If $\langle y,y \rangle =1$, output $y$.
\item 
If $\langle y,y \rangle \neq 1$, output
``there is no $y\in C$ with $\langle y,y\rangle=1$''.
\end{enumerate}
\end{algorithm}

\begin{prop}
\label{findingekmprop2}
Algorithm \ref{findingekmalg2} is a deterministic polynomial-time algorithm 
that, given a integral lattice $L$, an integer $m$ such that
$m \ge 2^{n/2} +1$ where $n=\rank(L)$, and $C\in L/mL$,
outputs all $y\in C$ with $\langle y,y\rangle=1$.
The number of such $y$ is $0$ or $1$.
\end{prop}

\begin{proof}
Suppose $x\in C$ with $\langle x,x \rangle=1$.
Since $x,y\in C$, there exists
$w\in L$ such that $x-y=mw$. Using the triangle inequality, we have
$$
m\langle w,w \rangle^{1/2} = \langle x-y, x-y \rangle^{1/2}
\le \langle x,x \rangle^{1/2} + \langle y,y \rangle^{1/2} 
< (1+2^{n/2})\langle x,x \rangle^{1/2} \le m,
$$
so $\langle w,w \rangle^{1/2} < 1$. 
Since $\langle w,w \rangle \in\Z_{\ge 0}$, we have $w=0$, and thus $y=x$.
If $\langle y,y \rangle \neq 1$, 
there is no $x\in C$ with $\langle x,x\rangle=1$.
\end{proof}

\section{$G$-lattices}
\label{backgroundsect2}

We introduce $G$-lattices and $G$-isomorphisms.
From now on, suppose that $G$ is a finite abelian group equipped with
a fixed element $u$ of order $2$, and
that $n = {\# G}/{2} \in\Z.$

\begin{defn}
\label{Sdef}
Let $S$ be a set of coset representatives of $G/\langle u\rangle$
(i.e., $\# S=n$ and $G = S \sqcup uS$), and 
for simplicity take $S$ so that $1\in S$.
\end{defn}

\begin{defn}
A {\bf $G$-lattice} is a lattice $L$ together with a group homomorphism
$f : G \to \Aut(L)$ such that  $f(u)= -1$.
For each $\sigma \in G$ and $x\in L$, define $\sigma x \in L$ by 
$\sigma x = f(\sigma)(x)$.
\end{defn}

The abelian group $G$ is specified by a multiplication table. The $G$-lattice $L$ is specified as a lattice along with, for each $\sigma \in G$, 
the matrix describing the action of $\sigma$ on $L$.

\begin{defn}
If $L$ and $M$ are $G$-lattices, then a {\bf $G$-isomorphism} 
is an isomorphism $\varphi : L \isom M$ of lattices
that respects the $G$-actions, i.e., 
$\varphi(\sigma x) = \sigma \varphi(x)$
for all $x\in L$ and $\sigma\in G$.
If such an isomorphism exists,
we say that $L$ and $M$ are {\bf $G$-isomorphic}, or isomorphic as $\G$-lattices.
\end{defn}

\section{The modified group ring $\Z\langle G\rangle$}
\label{modgpringsect}

We define a modified group ring $A\langle G\rangle$ whenever
$A$ is a commutative ring. We will usually take $A=\Z$, but
will also take $A=\Z/m\Z$ and $\Q$ and $\C$.

If $H$ is a group and $A$ is a commutative ring,
the group ring $A[H]$ is the set of formal sums
$\sum_{\sigma\in H} a_\sigma \sigma$ with $a_\sigma \in A$,
with addition defined by
$$
\sum_{\sigma\in H} a_\sigma \sigma + \sum_{\sigma\in H} b_\sigma \sigma
= \sum_{\sigma\in H} (a_\sigma + b_\sigma) \sigma
$$
and multiplication defined by
$$
(\sum_{\sigma\in H} a_\sigma \sigma)(\sum_{\tau\in H} b_\tau \tau)
= \sum_{\rho\in H} (\sum_{\sigma\tau =\rho} a_\sigma b_\tau) \rho.
$$
For example, if $H$ is a cyclic group of order $m$ 
and $h$ is a generator, then as rings we have 
$$
 \Z[X]/(X^m-1)  \cong  \Z[H] 
$$
via the map
$$
\sum_{i=0}^{m-1}a_iX^i  \mapsto  \sum_{i=0}^{m-1}a_ih^i.
$$

\begin{defn}
\label{modgpringdef}
If $A$ is a commutative ring, then
writing $1$ for the identity element of the group $G$,
we define the {\bf modified group ring}
$$
A\langle G\rangle = A[G]/(u+1).
$$
\end{defn}

Every $G$-lattice $L$ is a $\Z\langle G\rangle$-module, where one uses 
the $G$-action on $L$ to define $ax$ whenever $x\in L$ and 
$a\in \Z\langle G\rangle$.
This is why we consider $A\langle G\rangle$ rather than the standard group ring $A[G]$.
Considering groups equipped with an element of order $2$ 
allows us to include the cyclotomic rings $\Z[X]/(X^{2^k}+1)$
in our theory.

\begin{defn}
\label{tdef}
Define the {\bf scaled trace function} $t : A\langle G\rangle \to A$ by
$$
t(\sum_{\sigma\in G} a_\sigma \sigma) = a_{1}-a_u.
$$
\end{defn}
This is well defined since 
the restriction of $t$ to $(u+1)A[G]$ is $0$.
The map $t$ is the 
$A$-linear map 
satisfying $t(1)=1$,  $t(u)=-1$, and 
$t(\sigma)=0$ if $\sigma\in G$ and $\sigma\neq 1,u$.

\begin{defn}
\label{bardef}
For $a = \sum_{\sigma\in G} a_\sigma \sigma \in A\langle G\rangle$, 
define 
$$
\overline{a} = \sum_{\sigma\in G} a_\sigma \sigma^{-1}.
$$
\end{defn}

The map $a \mapsto \overline{a}$
is a ring automorphism of $A\langle G\rangle$. 
Since
$
\overline{\overline{a}} = a,$
it is an involution.
(An involution is a ring automorphism that is its own inverse.)
One can think of this map as mimicking complex conjugation
(cf.\ Lemma \ref{psimisclem}(i)).

\begin{rem}
If $L$ is a $G$-lattice and $x,y\in L$,
then 
$$
\langle \sigma x,\sigma y \rangle = \langle x,y\rangle
$$ 
for all $\sigma\in G$
by Definition \ref{latisomdefn}. 
It follows that 
$$
\langle a x, y \rangle = \langle x,\overline{a} y\rangle
$$
for all $a\in \Z\langle G\rangle$.
This ``hermitian'' property of the inner product is the main reason
for introducing the involution. 
\end{rem}

\begin{defn}
For $x,y\in \Z\langle G\rangle$ define
$\langle x,y\rangle_{\Z\langle G\rangle} = t(x\overline{y}).$
\end{defn}

Recall that $n = {\#G}/{2}$ and 
$S$ is a set of coset representatives of $G/\langle u\rangle$.
The following two results are straightforward.

\begin{lem}
\label{miscAlem}
Suppose $A$ is a commutative ring. Then:
\begin{enumerate}
\item
$A\langle G\rangle = \{ \sum_{\sigma \in S} a_\sigma \sigma : a_\sigma\in A \} =\bigoplus_{\sigma\in S}A\sigma$;
\item 
if $a = \sum_{\sigma \in S} a_\sigma \sigma \in A\langle G\rangle$,
then 
\begin{enumerate}
\item[{\rm{(a)}}]
$t(a) = a_1$, 
\item[{\rm{(b)}}]
$t(\bar{a}) = t(a)$, 
\item[{\rm{(c)}}]
$t(a\bar{a}) =  \sum_{\sigma \in S} a_\sigma^2$,
\item[{\rm{(d)}}]
$a =  \sum_{\sigma \in S} t(\sigma^{-1}a) \sigma$,
\item[{\rm{(e)}}]
if $t(ab)=0$ for all $b \in A\langle G\rangle$, then $a=0$.
\end{enumerate}
\end{enumerate}
\end{lem}

\begin{prop}
\begin{enumerate}[leftmargin=*]
\item
The additive group of the ring 
$\Z\langle G\rangle$ is a $G$-lattice of rank $n$,
with lattice structure defined
by $\langle \,\cdot\, , \,\cdot\, \rangle_{\Z\langle G\rangle}$ and $G$-action defined by 
$\sigma x = \sigma x$ where the right hand side is
ring multiplication in $\Z\langle G\rangle$.
\item  
As lattices, we have $\Z\langle G\rangle \cong \Z^n$.
\end{enumerate}
\end{prop}

\begin{defn}
We call $\Z\langle G\rangle$ the {\bf standard $G$-lattice}.
\end{defn}

The set $S$ of coset representatives for $G/\langle u\rangle$
is an orthonormal basis for the standard $G$-lattice.

\begin{ex}
\label{Gcyclicmodex}
Suppose $G=H \times \langle u\rangle$ with $H \cong \Z/n\Z$. 
Then
$$\Z\langle G\rangle \cong \Z[H] \cong \Z[X]/(X^n-1)$$
as rings and as lattices. 
When $n$ is odd (so $G$ is cyclic), then, 
sending $X$ to $-X$, we have
$$\Z\langle G\rangle \cong \Z[X]/(X^n-1) \cong \Z[X]/(X^n+1).$$
\end{ex}

\begin{ex}
If $G$ is cyclic, then 
$\Z\langle G\rangle \cong \Z[X]/(X^n+1)$, identifying $X$ with
a generator of $G$.
If $G$ is cyclic of order $2^r$, then 
$$\Z\langle G\rangle \cong \Z[X]/(X^{2^{r-1}}+1) \cong \Z[\zeta_{2^r}],$$ 
where $\zeta_{2^r}$ is
a primitive $2^r$-th root of unity.
\end{ex}

\begin{rem}
The ring $\Z\langle G\rangle$ is an integral domain
if and only if
$G$ is cyclic and $n$ is a power of $2$ (including $2^0=1$).
(If $g\in G$ is an element whose order is odd or $2$,
and $g \not\in \{ 1,u\}$, then $g-1$ is a zero divisor.)
\end{rem}

\section{The modified group ring over fields}
\label{recovergssect}
The main result of this section is Lemma \ref{psimisclem}, which we
will use repeatedly in the rest of the paper.
Recall that $G$ is a finite abelian group of order $2n$ equipped
with an element $u$ of order $2$.
If $R$ is a commutative ring, then a commutative $R$-algebra is a commutative ring
$A$ equipped with a ring homomorphism from $R$ to $A$.

If $K$ is a subfield of $\C$ and $E$ is a commutative $K$-algebra
with $\dim_K(E) < \infty$, let $\Phi_E$ denote the set of $K$-algebra
homomorphisms from $E$ to $\C$. Then $\C^{\Phi_E}$ is a $\C$-algebra
with coordinate-wise operations.
The next result is not only useful for studying modified group rings,
but also comes in handy in Proposition \ref{rootsofunitythm} below.

\begin{lem}
\label{KalglemBourb}
Suppose $K$ is a subfield of $\C$ and $E$ is a commutative $K$-algebra
with $\dim_K(E) < \infty$. Assume $\#\Phi_E = \dim_K(E)$. Then:
\begin{enumerate}
\item 
identifying $\Phi_E$ with 
$$
\{ \text{$\C$-algebra homomorphisms $E_\C = \C\otimes_K E \to \C$} \},
$$
 the  map
$E_\C \to \C^{\Phi_E}$, $x\mapsto (\varphi(x))_{\varphi\in\Phi_E}$ 
is an isomorphism of $\C$-algebras;
\item
$\bigcap_{\varphi\in\Phi_E}\ker(\varphi) = 0$ in $E$;
\item
there is a finite collection $\{ K_j\}_{j=1}^d$ of finite extension fields of $K$ 
such that 
$$
E \cong K_1\times\cdots\times K_d
$$ 
as $K$-algebras.
\end{enumerate}
\end{lem}

\begin{proof}
By the Corollaire to Proposition 1 in V.6.3 of \cite{BourbakiAlgII},
the set $\Phi_E$ is a $\C$-basis for $\Hom_K(E,\C) = \Hom_\C(E_\C,\C)$,
so the $\C$-algebra homomorphism in (i) is an isomorphism. Part (ii)
follows immediately from (i).

By Proposition 2 in V.6.3 of \cite{BourbakiAlgII}, the $K$-algebra $E$ is 
what Bourbaki calls an \'etale $K$-algebra, and (iii) then follows from
Theorem 4 in V.6.7 of \cite{BourbakiAlgII}.
\end{proof}

\begin{defn}
\label{Psidef}
Let $\Psi$ denote the set of ring homomorphisms from
$\Q\langle G\rangle$ to $\C$.
We identify $\Psi$ with the set
of $K$-algebra homomorphisms from $K\langle G\rangle$ to $\C$,
where $K$ is any subfield of $\C$.
The set $\Psi$ can also be identified with the set of
group homomorphisms $\psi : G \to\C^\ast$ such that $\psi(u)=-1$.
\end{defn}

We have $\#\Psi = n$, since
$\#\Hom(G,\C^\ast)=\# G=2n$ and the restriction map
$\Hom(G,\C^\ast)\to \Hom(\langle u\rangle,\C^\ast)$ is surjective.
This allows us to apply Lemma \ref{KalglemBourb} with $E=K\langle G\rangle$.
If $a\in \C\langle G\rangle$,
then $a$ acts on the $\C$-vector space $\C\langle G\rangle$
by multiplication, and for $\psi\in\Psi$ the $\psi(a)$ are the eigenvalues
for this linear transformation.
Lemma \ref{psimisclem}(ii) justifies thinking of the map
$t$ of Definition \ref{tdef} as a scaled trace function.

\begin{lem}
\label{psimisclem}
\begin{enumerate}[leftmargin=*]
\item 
If $\psi\in\Psi$, then 
$\overline{\psi(\alpha)} = \psi(\bar{\alpha})$
for all $\alpha\in \R\langle G\rangle$.
\item
If $a\in \C\langle G\rangle$,
 then $t(a) = \frac{1}{n}\sum_{\psi\in\Psi} \psi(a)$.
\item
If $K$ is a subfield of $\C$, then
$\bigcap_{\psi\in\Psi}\ker(\psi) = 0$ in $K\langle G\rangle$.
\item
The  map
$\C\langle G\rangle \to \C^{\Psi}$, $x\mapsto (\psi(x))_{\psi\in\Psi}$ 
is an isomorphism of $\C$-algebras.
\item
There are number fields $K_1,\ldots,K_d$ such that
$$
\Q\langle G\rangle \cong K_1\times\cdots\times K_d
$$ 
as $\Q$-algebras.
\item
Suppose 
$K$ is a subfield of $\C$ and $\alpha\in K\langle G\rangle$.
Then $\alpha\in K\langle G\rangle^\ast$ if and only if
$\psi(\alpha) \neq 0$ for all $\psi\in\Psi$.
\item
If $z\in\R\langle G\rangle$ is such that $\psi(z)\in\R$ for all $\psi\in\Psi$
and $\sum_{\psi\in\Psi}\psi(x\overline{x}z) \ge 0$ for all $x\in \R\langle G\rangle$,
then $\psi(z)\ge 0$ for all $\psi\in\Psi$.
\end{enumerate}
\end{lem}

\begin{proof}
For (i), since $G$ is finite, $\psi(\sigma)$ is a root of unity for all
$\sigma\in G$. Thus, 
$$
\overline{\psi(\sigma)} = \psi({\sigma})^{-1} = \psi({\sigma}^{-1}) = \psi(\bar{\sigma}).
$$
The $\R$-linearity of $\psi$ and of $\Aut(\C/\R)$ now imply (i).

We have 
$$
\frac{1}{n}\sum_{\psi\in\Psi} \psi(1) = 1 = t(1),
$$ 
and
$$
\frac{1}{n}\sum_{\psi\in\Psi} \psi(u) = -1 = t(u),
$$
and for each $\sigma \not\in \langle u\rangle$ we have
$$
\sum_{\psi\in\Psi} \psi(\sigma) = 
-\sum_{{\psi\in\Hom(G,\C^\ast)} \atop {\psi(u)=1}} \psi(\sigma) = 
-\sum_{\psi\in\Hom(G/\langle u\rangle,\C^\ast)} \psi(\sigma \text{ mod $\langle u\rangle ) = 0 = nt(\sigma)$}.
$$ 
Extending $\C$-linearly gives (ii).

If $K$ is a subfield of $\C$,
then $\#\Psi=n=\dim_K K\langle G\rangle$.
Thus we can apply Lemma \ref{KalglemBourb}, giving (iii), (iv), and (v).

By (iv) we have 
$\C\langle G\rangle^\ast \isom (\C^\ast)^{\Psi}$.
This gives (vi) when $K=\C$.
If $K$ is a subfield of $\C$ and
$x\in K\langle G\rangle \cap \C\langle G\rangle^\ast$
then multiplication by $x$ is an injective map from
$K\langle G\rangle$ to itself, so is also surjective, so
$x\in K\langle G\rangle^\ast$.
Thus  
$$
K\langle G\rangle^\ast = K\langle G\rangle \cap \C\langle G\rangle^\ast,
$$
and (vi) follows.

For (vii), applying Lemma \ref{KalglemBourb}(iii) with $K=\R$ gives
an $\R$-algebra isomorphism 
$$
\R\langle G\rangle \isom \R^{r} \times \C^{s}.
$$
The set $\Psi=\{\psi_j\}_{j=1}^{r+2s}$ consists of
the $r$ projection maps 
$\psi_j: \R\langle G\rangle \to \R\subset \C$ for $1 < j\le r$,
along with the $s$ projection maps 
$\psi_j: \R\langle G\rangle \to  \C$  
and their complex conjugates
$\psi_{s+j}=\overline{\psi_{j}}$ for $r+1\le j\le r+s$.
By (i), if $$x = (x_1,\ldots,x_r,y_1,\ldots,y_s)\in \R^{r} \times \C^{s},$$
then 
$$
\overline{x} = (x_1,\ldots,x_r,\overline{y_1},\ldots,\overline{y_s}).
$$
Taking  
$x$ to have $1$ in the $j$-th position and $0$ everywhere else, we have
$$
0 \le \sum_{\psi\in\Psi}\psi(x\overline{x}z) = 
\begin{cases}
\phantom{2}\psi_j(z) & \text{if $1\le j\le r$} \\
2\psi_j(z) & \text{otherwise,}
\end{cases}
$$ 
giving (vii).
\end{proof}

\section{Ideal lattices}
\label{ideallatsect}
As before, $G$ is a finite abelian group of order $2n$ equipped
with an element $u$ of order $2$.
Theorem \ref{Iwmisclem} below gives a way to view certain ideals
$I$ in $\Z\langle G\rangle$ as $G$-lattices, and
Theorem \ref{I1I2isom} characterizes the ones that are
$G$-isomorphic to $\Z\langle G\rangle$.

\begin{defn}
\label{fractidealdefn}
A {\em fractional $\Z\langle G\rangle$-ideal} is
a finitely generated $\Z\langle G\rangle$-module 
in $\Q\langle G\rangle$ that spans $\Q\langle G\rangle$ over $\Q$.
An {\em invertible} fractional $\Z\langle G\rangle$-ideal is a
fractional $\Z\langle G\rangle$-ideal $I$ such that there is a
fractional $\Z\langle G\rangle$-ideal $J$ with $IJ = \Z\langle G\rangle$,
where $IJ$ is the fractional $\Z\langle G\rangle$-ideal generated by
the products of elements from $I$ and $J$.
\end{defn}

\begin{thm}
\label{Iwmisclem}
Suppose $I \subset \Q\langle G\rangle$ is a fractional $\Z\langle G\rangle$-ideal 
and $w \in \Q\langle G\rangle$. Suppose 
that 
$
I\overline{I} \subset \Z\langle G\rangle \cdot w
$ 
and
$\psi(w) \in\R_{>0}$ for all 
$\psi \in\Psi$.
Then:
\begin{enumerate}
\item
$\overline{w}=w$;
\item 
$w\in \Q\langle G\rangle^\ast$;
\item
$I$ is a $G$-lattice, with $G$-action defined by multiplication in
$\Q\langle G\rangle$, and with lattice structure defined by
$$
\langle x,y\rangle_{I,w} = t\left(\frac{x\overline{y}}{w}\right),
$$
with $t$ as in Definition \ref{tdef}.
\end{enumerate}
\end{thm}

\begin{proof}
By Lemma \ref{psimisclem}(i) 
we have $$\psi(w) = \overline{\psi(w)} = \psi(\bar{w})$$
for all $\psi\in\Psi$. Now (i) follows from Lemma \ref{psimisclem}(iii). 
Lemma \ref{psimisclem}(vi) implies (ii).
Note that 
$
\frac{x\overline{y}}{w}\in \Z\langle G\rangle,
$ 
since $I\overline{I} \subset \Z\langle G\rangle \cdot w$.
Part (iii) now follows from (i) and (ii) of Lemma \ref{psimisclem}.
\end{proof}

\begin{notation}
\label{IwGlatdefn}
Let $I$ and $w$ be as in Theorem \ref{Iwmisclem}.
Define $L_{(I,w)}$ to be
the $G$-lattice $I$ with lattice structure defined by 
$\langle x,y\rangle_{I,w} = t({x\overline{y}}/{w})$.
\end{notation}

\begin{ex}
We have $L_{(\Z\langle G\rangle,1)} = \Z\langle G\rangle$.
\end{ex}

\begin{thm}
\label{I1I2isom}
Suppose that $I_1$ and $I_2$ are fractional $\Z\langle G\rangle$-ideals, 
that $w_1, w_2 \in \Q\langle G\rangle$, 
that 
$I_1\overline{I_1} \subset \Z\langle G\rangle \cdot w_1$ and 
$I_2\overline{I_2} \subset \Z\langle G\rangle \cdot w_2,$ 
and
that $\psi(w_1), \psi(w_2) \in\R_{>0}$ for all  
$\psi \in\Psi$.
Let $L_j = L_{(I_j,w_j)}$ for $j=1,2$.
Then sending $v$ to multiplication by $v$ gives a bijection from 
$$\{ v\in \Q\langle G\rangle :  I_1 = vI_2, w_1 = v\overline{v}w_2 \}
\quad \text{to} \quad  
\{ \text{$G$-isomorphisms $L_2 \isom L_1$} \}$$
and gives a bijection from
$$\{ v\in \Q\langle G\rangle : I_1 = v\Z\langle G\rangle, w_1=v\overline{v} \}
\quad \text{to} \quad  
\{ \text{$G$-isomorphisms 
$\Z\langle G\rangle \isom L_1$} \}.$$
In particular, 
$L_{1}$ is $G$-isomorphic to $\Z\langle G\rangle$ if and only if
there exists $v\in \Q\langle G\rangle$ such that $I_1 = (v)$ 
and $w_1=v\overline{v}$.
\end{thm}

\begin{proof}
Every $\Z\langle G\rangle$-module
isomorphism $\varphi : L_2 \isom L_1$ extends to a 
$\Q\langle G\rangle$-module isomorphism
$$
L_2 \otimes \Q = \Q\langle G\rangle \to L_1 \otimes \Q = \Q\langle G\rangle,
$$ 
and any such map is multiplication by some $v\in \Q\langle G\rangle^\ast$.
Conversely, for $v\in \Q\langle G\rangle$,
multiplication by $v$ 
defines a $\Z\langle G\rangle$-module isomorphism 
from $L_2$ to $L_1$
if and only if 
$I_1 = vI_2$.
When $I_1 = vI_2$, \ multiplication by $v$ 
is a $G$-isomorphism from $L_2$ to $L_1$
if and only if $w_1 = v\overline{v}w_2$;
this follows from Lemma \ref{miscAlem}(ii)(e), 
since
for all $a,b\in I_2$ we have
$$\langle a,b\rangle_{I_2,w_2} = t\left(\frac{a\overline{b}}{w_2}\right)
\quad \text{and} \quad
\langle av,bv\rangle_{I_1,w_1} = t\left(\frac{a\overline{b}v\overline{v}}{w_1}\right).$$
This gives the first desired bijection.
Taking $I_2=\Z\langle G\rangle$ and $w_2=1$ gives the second bijection.
\end{proof}

\begin{rem}
\label{GSrmk}
We next show how to recover the
Gentry-Szydlo algorithm from Theorem \ref{mainthm}.
The goal of the Gentry-Szydlo algorithm is to find a generator $v$
of a principal ideal $I$ of finite index in the ring $R=\Z[X]/(X^n-1)$, 
given $v\overline{v}$ and a $\Z$-basis for $I$.
Here, $n$ is an odd prime, and for 
$$
v=v(X)=\sum_{i=0}^{n-1}a_iX^i \in R,
$$
its ``reversal''  is
$$\overline{v} = 
v(X^{-1})=a_0 + \sum_{i=1}^{n-1}a_{n-i}X^i  \in R.$$
We take $G$ to be a cyclic group of order $2n$. Then 
$R \cong\Z\langle G\rangle$ as in Example \ref{Gcyclicmodex},
and we identify $R$ with $\Z\langle G\rangle$.
Let $w= v\overline{v} \in \Z\langle G\rangle$ and let
$L=L_{(I,w)}$ as in Notation \ref{IwGlatdefn}. 
Then $L$ is the ``implicit orthogonal lattice'' in \S 7.2 of \cite{GS}.
Once one knows $w$ and a $\Z$-basis for $I$, then one knows $L$. 
Theorem \ref{mainthm} produces a $G$-isomorphism
$\varphi : \Z\langle G\rangle \isom L$ in polynomial time, and thus 
(as in Theorem \ref{I1I2isom}) gives a generator 
$v = \varphi(1)$ in polynomial time. 
\end{rem}

\section{Invertible $G$-lattices}
\label{conceptsforpfsect}

Recall that $G$ is a finite abelian group of order $2n$, with a fixed
element $u$ of order $2$, and $S$ is a set of coset representatives for
$G/\langle u\rangle$.
In Definition \ref{invertdef2} we introduce the concept of an invertible
$G$-lattice. The inverse of such a lattice $L$ is the $G$-lattice
$\overline{L}$ given in Definition \ref{Lbardefn}.
  
\begin{defn}
\label{Lbardefn}
If $L$ is a $G$-lattice, then the
$G$-lattice $\overline{L}$ is a lattice equipped with a lattice
isomorphism 
$$
L\isom \overline{L}, \qquad
x \mapsto \overline{x}
$$ 
and a group homomorphism $G \to \Aut(\overline{L})$
defined by
$$
\sigma \overline{x} = \overline{\sigma^{-1}x}
$$ 
for all $\sigma\in G$ and $x\in L$,
i.e., 
$$
\overline{\sigma x} = \overline{\sigma}\, \overline{x}.
$$
\end{defn}

Existence follows by taking $\overline{L}$ to be $L$ with the appropriate
$G$-action. 
The $G$-lattice $\overline{L}$ is unique up to $G$-isomorphism, 
and we have 
$
\overline{\overline{L}} = L.
$

\begin{defn}
\label{dotproddef}
If $L$ is a $G$-lattice, 
define the {\bf lifted inner product} 
$$
\cdot: L \times \overline{L} \to  \Z\langle G\rangle
$$
by 
$$
x \cdot \overline{y} = 
\sum_{\sigma\in S} \langle x,\sigma y\rangle \sigma \in \Z\langle G\rangle.
$$
\end{defn}
This lifted inner product is independent of the choice of the set $S$, and is
$\Z\langle G\rangle$-bilinear; in fact, it extends $\Q$-linearly,
and for all $x, y \in L\otimes_\Z \Q$
and for all $a \in \Q\langle G\rangle$ we have
\begin{equation}
\label{tdotQeqn}
(ax) \cdot \overline{y} = x \cdot (a\overline{y}) = a(x \cdot \overline{y}),
\end{equation}
\begin{equation}
\label{tdoteqn}
\langle x, y\rangle=t(x \cdot \overline{y}), 
\end{equation}
and 
$
x \cdot \overline{y} = \overline{y \cdot \overline{x}}.
$

\begin{ex}
\label{LIwexdot}
If $I$, $w$, and $L_{(I,w)}$ are as in Theorem \ref{Iwmisclem}
and Notation \ref{IwGlatdefn}, then
$
\overline{L_{(I,w)}} = L_{(\overline{I},{w})},
$ 
and 
applying Lemma \ref{miscAlem}(ii)(d) with $a=\frac{x\overline{y}}{w}$
shows that 
$
x \cdot \overline{y} = \frac{x\overline{y}}{w}.
$
In particular, 
if $L=\Z\langle G\rangle$, then $\overline{L}=\Z\langle G\rangle$ 
with $\overline{\phantom{x}}$ having the same meaning as in
Definition \ref{bardef} for $A=\Z$, and with $\cdot$
being multiplication in $\Z\langle G\rangle$.
Note that when $w \neq 1$, ideals $I$ in $\Z\langle G\rangle$
do not inherit their lifted inner product from that of $\Z\langle G\rangle$.
\end{ex}

\begin{defn}
\label{invertdef2}
A $G$-lattice $L$ is {\bf invertible} if the following three conditions all hold:
\begin{enumerate}
\item $\rank(L) = n = \#G/2$;
\item
$L$ is unimodular (see Definition \ref{unimodulardef});
\item 
for each $m \in \Z_{>0}$ there exists $e_m \in L$ such that
$$
\{\sigma e_m + mL : \sigma\in G\}
$$ 
generates the abelian group $L/mL$.
\end{enumerate}
\end{defn}

It is clear from the definition that
invertibility is preserved under $G$-lattice isomorphisms. 
Definition \ref{invertdef2} implies that $L/mL$ is a
free $(\Z/m\Z)\langle G\rangle$-module of rank one for all $m>0$.
Given an ideal, it is a hard problem to decide if it is principal.
But checking (iii) of Definition \ref{invertdef2} is easy
algorithmically; see Algorithm \ref{elemprop2} below.

\begin{lem}
\label{stdGex}
If $L$ is a $G$-lattice and $L$ is $G$-isomorphic to the standard $G$-lattice, 
then $L$ is invertible. 
\end{lem}

\begin{proof}
Parts (i) and (ii) of Definition \ref{invertdef2} are easy.
For (iii), observe that the group $\Z\langle G\rangle$ is generated
by $\{ \sigma 1 : \sigma\in G\}$, so the group $L$ is generated
by $\{ \sigma e : \sigma\in G\}$ where $e$ is the image of $1$
under the isomorphism. Now let $e_m=e$ for all $m$.
\end{proof}

\section{Determining invertibility}
\label{detinvertalgsect}
Fix as before a finite abelian group $G$ of order $2n$ equipped with  
an element $u$ of order $2$.

Algorithm \ref{elemprop2} below determines whether a $G$-lattice
is invertible.
In Proposition \ref{elemprop2pf} we show that
Algorithm \ref{elemprop2} produces correct output
and runs in polynomial time.

In \cite{finmodspaper} we obtain a deterministic polynomial-time algorithm
that on input a finite commutative ring $R$ 
and a finite $R$-module $M$,
decides whether there exists $y\in M$ such that $M = Ry$, and if there is, 
finds such a $y$.
Applying this with $R=\Z\langle\G\rangle/(m)$ and
$M=L/mL$ gives the  algorithm in the following result.

\begin{prop}
\label{elemprop}
There is a deterministic  polynomial-time algorithm that,
given $\G$, $u$, a $\G$-lattice $L$, and $m\in\Z_{>0}$,  
decides whether there exists $e_m\in L$ such that 
$$
\{ \sigma e_m + mL : \sigma\in G\}
$$ 
generates $L/mL$ as an
abelian group,
and if there is, finds one.
\end{prop}

\begin{lem}
\label{inviiilem}
Suppose that $L$ is a $\G$-lattice, $m\in\Z_{>1}$, and $e\in L$.
Then:
\begin{enumerate}
\item
 $\{ \sigma e + mL : \sigma \in G\}$ generates $L/mL$
as an abelian group if and only if $L/(\Z\langle\G\rangle\cdot e)$ 
is finite of order coprime to $m$;
\item
if $\rk(L)=n$ and
$L/(\Z\langle\G\rangle\cdot e)$ is finite, then the map
$$
\Z\langle\G\rangle \to \Z\langle\G\rangle \cdot e, \qquad a\mapsto ae
$$
is an isomorphism of $\Z\langle\G\rangle$-modules.
\end{enumerate}
\end{lem}

\begin{proof}
The set $\{ \sigma e + mL : \sigma \in G\}$ generates $L/mL$
as an abelian group if and only if $L = \Z\langle\G\rangle e + mL$,
and if and only if 
multiplication by $m$ is surjective as a map from 
$L/(\Z\langle\G\rangle\cdot e)$ to itself.
Since $L/(\Z\langle\G\rangle\cdot e)$ is a finitely generated abelian group,
this holds if and only if $L/(\Z\langle\G\rangle\cdot e)$ is finite of order 
coprime to $m$. This gives (i).

Now suppose that $\rk(L)=n$ and
$L/(\Z\langle\G\rangle\cdot e)$ is finite.
The map in (ii) is clearly $\Z\langle\G\rangle$-linear and surjective.
Since $\Z\langle\G\rangle$ and $\Z\langle\G\rangle e$
both have rank $n$ over $\Z$, the map is injective.
\end{proof}

\begin{algorithm}
\label{elemprop2}
Given $\G$, $u$, and a $G$-lattice $L$,
the algorithm decides whether $L$ is invertible.

\begin{enumerate}
\item
If $\rank(L)\neq n$, output ``no'' (and stop).
\item
Compute the determinant of the Gram matrix for $L$.
If it is not $1$, output ``no'' (and stop).
\item
Use  Proposition \ref{elemprop} to determine if
$e_2$ (in the notation of Definition \ref{invertdef2}(iii)) exists. 
If no $e_2$ exists, output ``no'' and stop. 
Otherwise, use Proposition \ref{elemprop} to compute $e_2 \in L$.
\item
Compute the order $q$ of the group $L/(\Z\langle\G\rangle\cdot e_2)$.
\item
Use Proposition \ref{elemprop} to determine if
$e_q$ exists. If no $e_q$ exists, output ``no''.
Otherwise, output ``yes''.
\end{enumerate}
\end{algorithm}

\begin{prop}
\label{elemprop2pf}
Algorithm \ref{elemprop2} is a deterministic polynomial-time algorithm 
that, given $\G$, $u$, and a $\G$-lattice $L$,
decides whether $L$ is invertible.
\end{prop}

\begin{proof}
If Step (ii) outputs ``no'' then $L$ is not unimodular so it is not invertible.
We need to check Definition \ref{invertdef2}(iii) for
all $m$'s in polynomial time.
We show that it suffices to check two particular values of $m$, namely
$m=2$ and $q$.
By Lemma~\ref{inviiilem}(i),
the group $L/(\Z\langle\G\rangle\cdot e_2)$ is finite of odd order $q$.
If no $e_q$ exists, $L$ is not invertible.
If $e_q$ exists, then for {\em all} $m\in\Z_{>0}$ there exists
$e_m\in L$ that generates $L/mL$ as a $\Z\langle\G\rangle/(m)$-module,
as follows. 
We can reduce to $m$ being a prime power $p^t$, since if
$\gcd(m,m')=1$ then $L/mm'L$ is free of rank $1$ over
$\Z\langle\G\rangle/(mm')$ if and only if
$L/mL$ is free of rank $1$ over
$\Z\langle\G\rangle/(m)$ and
$L/m'L$ is free of rank $1$ over
$\Z\langle\G\rangle/(m')$.
Lemma~\ref{inviiilem}(i) now allows us to reduce to the case $m=p$. 
If $p$ does not divide $q$, we can take $e_{p} = e_2$.
If $p$ divides $q$, we can take $e_p=e_q$.
\end{proof}

\section{Equivalent conditions for invertibility}
\label{equivinvertdefsect}

In this section we prove Theorem \ref{invertequiv}, which
gives equivalent conditions for invertibility.

\begin{thm}
\label{invertequiv}
If $L$ is a $G$-lattice, 
then the following statements are equivalent:
\begin{enumerate}
\item[\rm (a)]
$L$ is invertible;
\item[\rm (b)]
the map 
$\varphi : L \otimes_{\Z\langle G\rangle} \overline{L} \to \Z\langle G\rangle$
defined by
$\varphi(x\otimes \overline{y}) = x \cdot \overline{y}$
is an isomorphism of $\Z\langle G\rangle$-modules, where
$\cdot$ is defined in Definition \ref{dotproddef};
\item[\rm (c)]
there is a  $\Z\langle G\rangle$-module $M$ such that
$L \otimes_{\Z\langle G\rangle} M$ and $\Z\langle G\rangle$
are isomorphic as $\Z\langle G\rangle$-modules, and as a lattice $L$ is
unimodular;
\item[\rm (d)]
$L$ is $G$-isomorphic to 
$L_{(I,w)}$
for some fractional $\Z\langle G\rangle$-ideal $I$  and 
some $w \in \Q\langle G\rangle^\ast$ such that
$I\overline{I} = \Z\langle G\rangle \cdot w$ and
$\psi(w) \in\R_{>0}$ for all 
$\psi \in\Psi$, with $L_{(I,w)}$  as in
Notation \ref{IwGlatdefn}. 
\end{enumerate}
\end{thm}

We will prove Theorem \ref{invertequiv} in a series of lemmas.
The equivalence of (a) and (c) says that being invertible as
a $G$-lattice is equivalent to being both unimodular as a lattice and invertible
as a $\Z\langle G\rangle$-module.

\begin{defn}
Suppose $R$ is a commutative ring.
An $R$-module is {\bf projective} if it is a direct summand of a free $R$-module.
An $R$-module $M$ is {\bf flat} if
whenever $N_1 \hookrightarrow N_2$ is an injection of $R$-modules,
then the induced map 
$$
M\otimes_{R} N_1 \to M\otimes_{R} N_2
$$
is injective.
\end{defn}

\begin{lem}
\label{flatproj}
Suppose that $L$ is a $\Z$-free $\Z\langle G\rangle$-module of rank $\#G/2$,
and for each $m \in \Z_{>0}$ there exists $e_m \in L$ such that
$$
\{\sigma e_m + mL : \sigma\in G\}
$$ 
generates the abelian group $L/mL$.
Then:
\begin{enumerate}
\item
there is a $\Z\langle G\rangle$-module $M$ such that
$L \oplus M$ and $\Z\langle G\rangle \oplus \Z\langle G\rangle$ are
isomorphic as $\Z\langle G\rangle$-modules, and
\item
$L$ is projective and flat as a $\Z\langle G\rangle$-module.
\end{enumerate}
\end{lem}

\begin{proof}
Let 
$q = (L : \Z\langle G\rangle e_2).$
By Lemma \ref{inviiilem}(i), we have that $q$ is finite and odd. 
Let 
$r = (L : \Z\langle G\rangle e_q).$ 
By Lemma \ref{inviiilem}(i), we have that $r$ is finite and coprime to $q$.
Take $a,b\in \Z$ such that $ar + bq = 1$.
Let 
$
N = \Z\langle G\rangle e_2 \oplus \Z\langle G\rangle e_q. 
$
By Lemma \ref{inviiilem}(ii) we have
$N \cong \Z\langle G\rangle \oplus \Z\langle G \rangle$ 
as $\Z\langle G\rangle$-modules.
Define 
$$
p : N \to L \quad \text{ by } \quad (x,y) \mapsto x+y
$$ 
and 
$$
s : L \to N \quad \text{ by } \quad x \mapsto (bqx,arx).
$$
Then $p\circ s$
is the identity on $L$. 
Thus, $$L \oplus \ker(p) \cong N
\cong \Z\langle G\rangle \oplus \Z\langle G\rangle$$
 as $\Z\langle G\rangle$-modules.
 So (i) holds with $M = \ker(p)$.
Since $L$ is a direct summand of a free module,
 $L$ is projective.
All projective modules are flat (by Example (1) in I.2.4 of \cite{BourbCommAlg}).
\end{proof}

Recall that the notions of fractional $\Z\langle G\rangle$-ideal
and invertible fractional $\Z\langle G\rangle$-ideal
were defined in Definition \ref{fractidealdefn}.

\begin{lem}
\label{fractideallem}
If $I$ is an invertible
fractional $\Z\langle G\rangle$-ideal, then:
\begin{enumerate}
\item
if $m\in\Z_{>0}$, 
then $I/mI$ is isomorphic to $(\Z/m\Z)\langle G\rangle$ as a $\Z\langle G\rangle$-module;
\item
$I$ is flat;
\item
if $I'$ is a fractional $\Z\langle G\rangle$-ideal, then the
natural surjective map 
$$
I\otimes_{\Z\langle G\rangle} I' \to II'
$$ 
is an isomorphism.
\end{enumerate}
\end{lem}

\begin{proof}
Since $I$ is an invertible
fractional $\Z\langle G\rangle$-ideal, there is a
fractional $\Z\langle G\rangle$-ideal $J$ such that
$IJ = \Z\langle G\rangle$.
Let $\FFF$ denote the partially ordered 
set of fractional $\Z\langle G\rangle$-ideals.
The maps from $\FFF$ to itself defined by $f_1: N \mapsto NI$ and
$f_2: N \mapsto NJ$
are inverse bijections that preserve inclusions.
Since $f_1(\Z\langle G\rangle) = I$, 
it follows that the maximal $\Z\langle G\rangle$-submodules 
of $I$
are exactly the $\m I$ such that 
$\m$ is a maximal ideal of $\Z\langle G\rangle$. 
By the Chinese Remainder Theorem, the map
$I \to \prod_{\m} I/\m I$ 
is surjective,
where the product runs over the (finitely many) maximal ideals $\m$ that contain $m$.
It follows that there exists $x\in I$ 
that is not contained in any $\m I$.
Since $\Z\langle G\rangle x + m I$ is a fractional ideal that is not contained in any
proper submodule of $I$, it equals $I$. Thus,
$I/mI$ is isomorphic to $(\Z/m\Z)\langle G\rangle$ as a $\Z\langle G\rangle$-module.
This proves (i).

For (ii), apply (i) and Lemma \ref{flatproj}(ii).

Since $I$ is flat, the natural map
$$
I \otimes_{\Z\langle G\rangle} I' \to I \otimes_{\Z\langle G\rangle} \Q\langle G\rangle
\cong I \otimes_{\Z\langle G\rangle} \Z\langle G\rangle \otimes_{\Z} \Q 
\cong I \otimes_{\Z} \Q = \Q\langle G\rangle
$$
is injective, giving (iii).
\end{proof}

Let $L_\Q = L \otimes_\Z \Q$. Then the inner product $\langle \,\, , \,\, \rangle$
on $L$ extends $\Q$-bilinearly to a $\Q$-bilinear, symmetric, positive definite 
inner product on $L_\Q$, and the lifted inner product
$\cdot$ extends $\Q$-bilinearly to a $\Q\langle G\rangle$-bilinear map 
$\cdot$ from
$L_\Q \times \overline{L_\Q}$ to $\Q\langle G\rangle$.

\begin{lem}
\label{aimpliesdlem}
Suppose $L$ is an invertible $G$-lattice. Then
$L_\Q = \Q\langle G\rangle \gamma$ for some $\gamma \in L_\Q$.
For such a $\gamma$, letting 
$z = \gamma\cdot\overline{\gamma} \in \Q\langle G\rangle$ we have:
\begin{enumerate}
\item
$\langle a\gamma,b\gamma\rangle = t(a\overline{b}z)$
for all $a,b\in \Q\langle G\rangle$,
\item
$z  \in \Q\langle G\rangle^\ast$,
\item
for all $\psi\in\Psi$ we have $\psi(z)\in\R_{>0}$,
\item
$L \cdot\overline{L} = \Z\langle G\rangle$,
\item
if $I = \{ x\in \Q\langle G\rangle : x\gamma\in L\}$, then
$I\overline{I} = \Z\langle G\rangle z^{-1}$ and as $G$-lattices
we have 
$L_{(I,z^{-1})} \cong L$.
\end{enumerate}
\end{lem}

\begin{proof}
By Definition \ref{invertdef2}(iii) 
and Lemma \ref{inviiilem}(i) 
we have that for all $m\in\Z_{>1}$ there exists $e_m\in L$
such that the index 
$i(m)=(L : \Z\langle G\rangle e_m)$ 
is finite
and coprime to $m$.
It follows that $\Q\langle G\rangle \cong L_\Q$ as $\Q\langle G\rangle$-modules.
Let $\gamma\in L_\Q$ be the image of $1$ under such an isomorphism
$\Q\langle G\rangle \isom L_\Q$.
Then 
$L_\Q = \Q\langle G\rangle \gamma.$
Let 
$$
z = \gamma\cdot\overline{\gamma} \in \Q\langle G\rangle.
$$

By \eqref{tdotQeqn} and  \eqref{tdoteqn},
for all $a,b\in \Q\langle G\rangle$ we have 
$$(a\gamma)\cdot (\overline{b\gamma}) =
a(\gamma\cdot (\overline{b}\overline{\gamma})) =
a\overline{b}(\gamma\cdot \overline{\gamma}) =
 a\overline{b}z$$
and thus 
$$
\langle a\gamma,b\gamma\rangle =
t((a\gamma)\cdot(\overline{b\gamma})) = t(a\overline{b}z),
$$ 
giving (i). 
Since the inner product on $L_\Q$ is symmetric, using Lemma \ref{miscAlem}(ii)(e)
we have $\bar{z}=z$.
Thus for all $\psi\in\Psi$ we have $$\psi(z) = \psi(\bar{z}) = \overline{\psi(z)}$$
by Lemma \ref{psimisclem}(i), so $\psi(z) \in \R$.
For all $a\in \Q\langle G\rangle$ we have 
$$0 \le \langle a\gamma,a\gamma\rangle =
t(a\overline{a}z) = \frac{1}{n}\sum_{\psi\in\Psi} \psi(a\overline{a}z)$$ by Lemma \ref{psimisclem}(ii). 
By Lemma \ref{psimisclem}(vii) it follows that $\psi(z)\ge 0$ for all
$\psi\in\Psi$.
If $a\in \Q\langle G\rangle$ and $za=0$, then 
$$\langle a\gamma,a\gamma\rangle=t(a\overline{a}z)=0,$$ so $a=0$.
Therefore multiplication by $z$ is an injective, and thus surjective,
map from $\Q\langle G\rangle$ to itself. Thus $z\in \Q\langle G\rangle^\ast$
and $\psi(z)\in\R_{> 0}$ for all $\psi\in\Psi$, by Lemma \ref{psimisclem}(vi).
This gives (ii) and (iii).

Define $$L^{-1} = 
\{ \overline{y} \in \overline{L}_\Q : L\cdot\overline{y} \subset \Z\langle G\rangle \}$$
and let $m\in\Z_{>1}.$
We have 
$$
L \supset \Z\langle G\rangle e_m \supset i(m)L,
$$ 
so 
$e_m\in \Q\langle G\rangle^\ast\gamma$ and therefore
$e_m\cdot\overline{e_m}\in \Q\langle G\rangle^\ast$.
Now $$i(m)(e_m\cdot\overline{e_m})^{-1}\overline{e_m} \in L^{-1},$$
because for all $x\in L$ one has 
$$i(m)x\cdot(e_m\cdot\overline{e_m})^{-1}\overline{e_m} \subset
\Z\langle G\rangle e_m\cdot(e_m\cdot\overline{e_m})^{-1}\overline{e_m} = 
\Z\langle G\rangle.$$
Therefore $$i(m) = e_m\cdot i(m)(e_m\cdot\overline{e_m})^{-1}\overline{e_m}
\in L\cdot L^{-1} \subset \Z\langle G\rangle.$$
This is true for all $m\in\Z_{>1}$, so $1\in L\cdot L^{-1}$ and
$L\cdot L^{-1} = \Z\langle G\rangle$.

Now for $\overline{y} \in \overline{L}_\Q$ one has
$\overline{y} \in \overline{L}$ if and only if $y\in L$, 
if and only if for all $x\in L$ one has $\langle x,y\rangle\in\Z$,
if and only if for all $x\in L$ and $\sigma\in G$ one has 
$\langle x,\sigma y\rangle = \langle \sigma^{-1}x, y\rangle \in\Z$,
if and only if for all $x\in L$ one has $x\cdot \overline{y}\in\Z\langle G\rangle$,
if and only if $\overline{y} \in {L}^{-1}$.
So $\overline{L} = {L}^{-1}$. Thus $L\cdot\overline{L} = \Z\langle G\rangle$,
giving (iv).

If $I\subset \Q\langle G\rangle$ is such that $L=I\gamma$,
then $I\isom L$, $x\mapsto x\gamma$ as $\Z\langle G\rangle$-modules.
Then $$\Z\langle G\rangle = L\cdot\overline{L} = 
I\overline{I}\gamma\cdot\overline{\gamma}
= I\overline{I}z,$$ so 
$I\overline{I} = \Z\langle G\rangle z^{-1}.$
Now $$\langle x\gamma,y\gamma\rangle = t(x\gamma\cdot \overline{y\gamma})=
t(x \overline{y}z) = \langle x,y\rangle_{I,z^{-1}}$$ for all $x,y\in I$. Thus, 
$L_{(I,z^{-1})} \cong L$ as $G$-lattices.
This gives (v).
\end{proof}

We are now ready to prove Theorem \ref{invertequiv}.

For (a) $\Rightarrow$ (d), apply Lemma \ref{aimpliesdlem} with $w=z^{-1}$.

For (d) $\Rightarrow$ (b), by (d) we have  
$L\otimes_{\Z\langle G\rangle}\overline{L} = 
I\otimes_{\Z\langle G\rangle}\overline{I}.$
Using Lemma \ref{fractideallem}(iii) we have that the  composition
$$
I\otimes \overline{I} \isom I\overline{I}=\Z\langle G\rangle w \isom \Z\langle G\rangle
$$
is an isomorphism,
where the first map sends $x\otimes y$ to $x\overline{y}$ and
the last map sends $\alpha$ to $\alpha/w$.
Since $x \cdot \overline{y} = {x\overline{y}}/{w}$,
this gives (b).

For (b) $\Rightarrow$ (c),
suppose (b) holds, i.e., the map 
$$\varphi : L \otimes_{\Z\langle G\rangle} \overline{L} \to \Z\langle G\rangle, \quad
x\otimes \overline{y} \mapsto x \cdot \overline{y}$$
is an isomorphism of $\Z\langle G\rangle$-modules.
Then $L$ is unimodular, as follows.
Consider the maps:
$$
L \to \Hom_{\Z\langle G\rangle}(\overline{L},\Z\langle G\rangle)
\to \Hom(\overline{L},\Z)
\to \Hom({L},\Z)
$$
where the left-hand map is the $\Z\langle G\rangle$-module isomorphism
induced by $\varphi$, defined by $x \mapsto (\bar{y} \mapsto x \cdot \overline{y})$,
the middle map is $f\mapsto t\circ f$, 
and the right-hand map is $g \mapsto (y\mapsto g(\bar{y}))$.
The latter two maps are group isomorphisms; for the middle map
note that its inverse is 
$$\hat{f}\mapsto (\overline{x} \mapsto \sum_{\sigma\in S} \hat{f}(\sigma^{-1}\overline{x})\sigma).$$
The composition, which takes $x$ to 
$$
(y \mapsto t(x \cdot \overline{y}) = \langle x,y\rangle),
$$ 
is therefore a bijection, so $L$ is unimodular.
Then (c) holds by taking $M=\overline{L}$.

For (c) $\Rightarrow$ (a), by Lemma \ref{psimisclem}(v)
we have 
$\Q\langle G\rangle \cong \prod_{j\in J}K_j$ 
with $\# J < \infty$
and fields $K_j$. 
Each $\Q\langle G\rangle$-module $V$ is $V = \prod_{j\in J}V_j$
with each $V_j$ a $K_j$-vector space.
With $V = L\otimes_\Z \Q$ and $W = M\otimes_\Z \Q$
we have 
$$
\prod_{j\in J} (V_j\otimes_{K_j} W_j)  = V\otimes_{\Q\langle G\rangle} W  
\cong \Q\langle G\rangle \cong \prod_{j}K_j.
$$
This holds if and only if for all $j$ we have 
$$
(\dim_{K_j} V_j)(\dim_{K_j} W_j) = 1,
$$ 
which holds
if and only if for all $j$ we have 
$$
\dim_{K_j} V_j = \dim_{K_j} W_j = 1.
$$
This holds if and only if $V \cong W \cong \Q\langle G\rangle$
as $\Q\langle G\rangle$-modules.
Thus, $L$ and $M$ may be viewed as fractional $\Z\langle G\rangle$-ideals in
$\Q\langle G\rangle$, and $LM$ is principal, so $L$ and $M$ are invertible
fractional $\Z\langle G\rangle$-ideals.
By Lemma \ref{fractideallem}(i), if $I$ is an invertible
fractional $\Z\langle G\rangle$-ideal, then $I/mI$ is cyclic as
a $\Z\langle G\rangle$-module, for every positive integer $m$.
Thus $L/mL$ is cyclic as
a $\Z\langle G\rangle$-module, so (a) holds.

This concludes the proof of Theorem \ref{invertequiv}.

\section{Short vectors in invertible lattices}
\label{shorttensorsect}
Recall that $G$ is a group of order $2n$ equipped with an element
$u$ of order $2$. 
The main result of this section is Theorem \ref{shortthm}, which shows
in particular that a $G$-lattice is $G$-isomorphic to the standard $G$-lattice
if and only if it is invertible and has a short vector 
(i.e., a vector of length $1$).

\begin{defn}
We will say that a vector $e$ in an integral lattice $L$ is {\bf short} if 
$\langle e,e \rangle = 1.$
\end{defn}

\begin{ex}
\label{shortex}
The short vectors in the standard lattice of rank $n$ are the $2n$
signed standard basis vectors 
$$
\{ (0,\ldots,0,\pm 1,0,\ldots,0) \}.
$$
Thus, the set of short vectors in $\Z\langle G\rangle$ is $G$.
\end{ex}

\begin{prop}
\label{eshortlem}
Suppose $L$ is an invertible $G$-lattice. Then:
\begin{enumerate}

\item
if $e$ is short, then
$\{ \sigma\in G : \sigma e = e\} = \{ 1\}$;
\item
if $e$ is short, then
$$
\langle e,\sigma e\rangle =
\begin{cases}
\phantom{-}1 & \text{ if $\sigma =1$,} \\
-1 & \text{ if $\sigma = u$,} \\ 
\phantom{-}0  & \text{ for all other $\sigma\in G$};
\end{cases}
$$
\item
 $e \in L$ is short if
and only if $e \cdot \overline{e} = 1$, with inner product $\cdot$
defined in Definition \ref{dotproddef}.
\end{enumerate}
\end{prop}

\begin{proof}
Suppose  $e \in L$ is short.
Let 
$$
H = \{ \sigma\in G : \sigma e = e\}.
$$ 
For all $\sigma\in G$, by the Cauchy-Schwarz inequality we have
$$
|\langle e,\sigma e\rangle | \le 
(\langle e, e\rangle \langle \sigma e,\sigma e\rangle)^{1/2}
= \langle e, e\rangle = 1,
$$
and 
$|\langle e,\sigma e\rangle | = 1$ if and only if
$e$ and $\sigma e$ lie on the same line through $0$.
Thus 
$$
\langle e,\sigma e\rangle \in \{ 1,0,-1\}.
$$
Then $\langle e,\sigma e\rangle = 1$ if and only if 
$\sigma\in H$.  
Also, $\langle e,\sigma e\rangle = -1$ if and only if 
$\sigma e = -e$ if and only if $\sigma\in Hu$.
Otherwise, $\langle e,\sigma e\rangle = 0$. 
Thus for (i,ii), it suffices to prove $H = \{ 1\}$.
Let $m=\#H$.

Let $T$ be a set of coset representatives for $G$ mod $H\langle u\rangle$
and let $S=T\cdot H$, a set of coset representatives for $G$ mod $\langle u\rangle$.
If $$a=\sum_{\sigma\in S}a_\sigma \sigma \in (\Z/m\Z)\langle G\rangle$$
is fixed by $H$,
then $a_{\tau\sigma}= a_\sigma$ 
for all $\sigma\in S$ and $\tau\in H$, so
$$a \in \left(\sum_{\tau\in H} \tau\right)(\Z/m\Z)\langle G\rangle.$$
By Definition \ref{invertdef2}, Theorem \ref{invertequiv}, and Lemma \ref{fractideallem},
there is a $\Z[H]$-module isomorphism
$$
L/mL \cong (\Z/m\Z)\langle G\rangle.
$$
Since $e + mL$ is fixed by $H$, we have
$$e + mL \in \left(\sum_{\tau\in H} \tau\right)(L/mL),$$ so 
$e_m\in mL + (\sum_{\tau\in H} \tau)L$.
Write
$$
e= m\varepsilon_1 + \left(\sum_{\tau\in H} \tau\right)\varepsilon_2
$$
with $\varepsilon_1, \varepsilon_2\in L$.
Since 
$$
\langle e, \tau\varepsilon_2\rangle = \langle \tau e, \tau\varepsilon_2\rangle
= \langle e, \varepsilon_2\rangle
$$ 
for all $\tau\in H$, we have
$$
1 = \langle e, e\rangle = 
m\langle e, \varepsilon_1\rangle + \sum_{\tau\in H} \langle e, \tau\varepsilon_2\rangle
= m\langle e, \varepsilon_1 + \varepsilon_2\rangle \equiv 0 \mod{m}.
$$
Thus, $m=1$ as desired. Part (iii) follows directly from (ii) and
Definition \ref{dotproddef}.
\end{proof}

This enables us to prove the following result.

\begin{thm}
\label{shortthm}
Suppose $L$ is a $G$-lattice. Then:
\begin{enumerate}

\item
if $L$ is invertible, then the map 
$$
\{ G\text{-isomorphisms $\Z\langle G\rangle \to L \} \to \{$short vectors of $L$\}}
$$
that sends $f$ to $f(1)$
is bijective;
\item
if $e \in L$ is short and $L$ is invertible, then 
$\{ \sigma e : \sigma \in G\}$ generates the abelian group $L$;
\item
 $L$ is $G$-isomorphic to $\Z\langle G\rangle$
if and only if $L$ is invertible and has a short vector;
\item
if $e \in L$ is short and $L$ is invertible, 
then the map 
$$
G \to \{\text{short vectors of $L$\}}, \qquad
\sigma \mapsto \sigma e
$$ 
is bijective.
\end{enumerate}
\end{thm}

\begin{proof}
For (i), that $f(1)$ is short is clear.
Injectivity of the map $f\mapsto f(1)$ follows from 
$\Z\langle G\rangle$-linearity of $G$-isomorphisms.
For surjectivity, suppose $e\in L$ is short.
Proposition \ref{eshortlem}(ii) says that $\{ \sigma e\}_{\sigma\in S}$
is an orthonormal basis for $L$.
Parts (ii) and (i) now follow, where the $G$-isomorphism $f$
is defined by $x \mapsto xe$ for all $x\in \Z\langle G\rangle$.
Part (iii) follows from (i) and Lemma \ref{stdGex}.
Part (iv) is trivial for $\Z\langle G\rangle$, and
$L$ is $G$-isomorphic to $\Z\langle G\rangle$, so we have (iv).
\end{proof}

\section{Tensor products of $G$-lattices}
\label{tensoringsect}

Recall that $G$ is a finite abelian group with an element $u$ of order $2$.
We will define the tensor product of  invertible $G$-lattices, 
and derive some properties.
See \cite{AtiyahMcD,LangAlg} for background on tensor products.
 
\begin{defn}
\label{tensoringGdef}
Suppose that $L$ and $M$ are invertible $G$-lattices.
Define the $\Z\langle G\rangle$-bilinear map  
$$\cdot: (L \otimes_{\Z\langle G\rangle} M) \times 
(\overline{L} \otimes_{\Z\langle G\rangle} \overline{M})
\to \Z\langle G\rangle, \quad (a, \overline{b}) \mapsto a \cdot \overline{b}$$ 
by letting
$$(x \otimes v) \cdot (\overline{y} \otimes \overline{w}) = 
(x \cdot \overline{y})(v \cdot \overline{w})$$
for all $x, y \in L$ and $v, w \in M$ and extending
$\Z\langle G\rangle$-bilinearly.
Take 
$$
\overline{L \otimes_{\Z\langle G\rangle} M}
$$ 
to be
$\overline{L} \otimes_{\Z\langle G\rangle} \overline{M}$, with
$$\overline{x \otimes v} = \overline{x} \otimes \overline{v}.$$ 
\end{defn}

\begin{ex}
\label{tensorLdef}
Let $L=L_{(I_1,w_1)}$ and $M=L_{(I_2,w_2)}$ where $I_1, I_2$ are
fractional $\Z\langle G\rangle$-ideals,   
$w_1,w_2 \in \Q\langle G\rangle^\ast$ are such that
$\psi(w_i) \in\R_{>0}$ for all $\psi \in\Psi$, and
$I_i\overline{I_i} = \Z\langle G\rangle  w_i$ for $i=1,2$. 
Then $L \otimes_{\Z\langle G\rangle} M$ may be identified
with $I_1I_2$ via  Lemma \ref{fractideallem},
and $\overline{L \otimes_{\Z\langle G\rangle} M}$ 
may be identified with $\overline{I_1I_2}$,
and the dot product 
$$
I_1I_2 \times \overline{I_1I_2} \to \Z\langle G\rangle
$$
from Definition \ref{tensoringGdef}
becomes $a\cdot\overline{b} = a\overline{b}/(w_1w_2)$
as in Example \ref{LIwexdot}.
This is precisely the lifted inner product of the $G$-lattice
$L_{(I_1I_2,w_1w_2)}$ (which is invertible by Theorem \ref{invertequiv}).
We thus have
\begin{equation}
\label{LItensoreq}
L_{(I_1,w_1)} \otimes_{\Z\langle G\rangle} L_{(I_2,w_2)} = L_{(I_1I_2,w_1w_2)}.
\end{equation}
\end{ex}

\begin{thm}
\label{tensoringGlattices}
Let $L$ and $M$ be invertible $G$-lattices.
Then $L \otimes_{\Z\langle G\rangle} M$ is an invertible $G$-lattice
with inner product 
$$
\langle a,b\rangle = t(a\cdot\overline{b}),
$$
where the dot product is defined in Definition \ref{tensoringGdef}
and equals the lifted inner product for this $G$-lattice. 
\end{thm}

\begin{proof}
By Theorem \ref{invertequiv} we may assume  that
$L=L_{(I_1,w_1)}$ and $M=L_{(I_2,w_2)}$ where $I_1, I_2$ are
fractional $\Z\langle G\rangle$-ideals,   
$w_1,w_2 \in \Q\langle G\rangle^\ast$ are such that
$\psi(w_i) \in\R_{>0}$ for all $\psi \in\Psi$, and
$I_i\overline{I_i} = \Z\langle G\rangle  w_i$ for $i=1,2$.
In this case, we already checked the theorem in Example \ref{tensorLdef}.
\end{proof}

\begin{prop}
\label{invertequivcor}
Suppose that $L$, $M$, and $N$ are invertible $G$-lattices. 
Then we have the following $G$-isomorphisms:
\begin{enumerate}
\item 
$L \otimes_{\Z\langle G\rangle} M \cong M \otimes_{\Z\langle G\rangle} L$,
\item
$(L \otimes_{\Z\langle G\rangle} M) \otimes_{\Z\langle G\rangle} N \cong 
L \otimes_{\Z\langle G\rangle} (M \otimes_{\Z\langle G\rangle} N)$,
\item
$L \otimes_{\Z\langle G\rangle} \Z\langle G\rangle\cong L$,
\item
$L \otimes_{\Z\langle G\rangle} \overline{L} \cong \Z\langle G\rangle$.
\end{enumerate}
\end{prop}

\begin{proof}
By Theorem \ref{invertequiv} we may reduce to the case where the invertible $G$-lattices
are of the form $L_{(I,w)}$.
Then \eqref{LItensoreq} immediately gives (i) and (ii).
For (iii) and (iv),
note that $\Z\langle G\rangle = L_{(\Z\langle G\rangle,1)}$,
and if $L = L_{(I,w)}$ then 
$$\overline{L} \cong L_{(\overline{I},w)} \cong L_{(\overline{I}w^{-1},w^{-1})} = 
L_{({I}^{-1},w^{-1})}.$$
\end{proof}

\begin{rem}
One can extend parts (i), (ii), and (iii) of Proposition \ref{invertequivcor}
to general $G$-lattices, by replacing $L \otimes_{\Z\langle G\rangle} M$
by its image in $L_\Q \otimes_{\Q\langle G\rangle} M_\Q$.
That image is a $G$-lattice with lifted inner product given by the same
formula.
\end{rem}

\section{The Witt-Picard group}
\label{WPsect}

This section, which is mostly a digression, is devoted to
what we call the Witt-Picard group $\WPic_{\Z\langle G\rangle}$.
The results of this section are not directly used later, with
the exception of the proof of Theorem \ref{WPgpfin}, but
it may be said that the properties of $\WPic_{\Z\langle G\rangle}$,
in particular its finiteness, are what makes our algorithms possible.
Also, several of our results admit an attractive reformulation
in terms of $\WPic_{\Z\langle G\rangle}$.

As before, $G$ is a finite abelian group of order $2n$ equipped with
an element $u$ of order $2$.

\begin{defn}
\label{WPgpdefn}
We define  
$$\WPic_{\Z\langle G\rangle} 
=\{ [L] : L \text{ is an invertible $G$-lattice} \},$$
where the symbols $[L]$ are chosen so that $[L]=[M]$ if and only if
$L$ and $M$ are $G$-isomorphic.
\end{defn}

\begin{thm}
\label{WPicabgpthm}
The set $\WPic_{\Z\langle G\rangle}$
is an abelian group, with group operation defined by
$$
[L]\cdot[M]=[L \otimes_{\Z\langle G\rangle} M],
$$
with identity element $[\Z\langle G\rangle]$, and with
$$
[L]^{-1}=[\overline{L}].
$$
\end{thm}

\begin{proof}
This follows immediately from Theorem \ref{tensoringGlattices}
and Proposition \ref{invertequivcor}.
\end{proof}

\begin{cor}
\label{invertequivcor2}
Suppose that $L$ and $M$ are invertible $G$-lattices.
Then $L$ and $M$ are $G$-isomorphic if and only if 
$L \otimes_{\Z\langle G\rangle} \overline{M}$
and $\Z\langle G\rangle$ are $G$-isomorphic.
\end{cor}

\begin{proof}
This follows immediately from Theorem \ref{WPicabgpthm}.
More precisely,
\begin{align*}
L\cong_G M & \iff [L] = [M] \\
& \iff [L][M]^{-1}=1=[\Z\langle G\rangle] \\
& \iff [L \otimes_{\Z\langle G\rangle} \overline{M}]=[\Z\langle G\rangle] \\
& \iff L \otimes_{\Z\langle G\rangle} \overline{M} \cong_G \Z\langle G\rangle
\end{align*}
where $\cong_G$ means $G$-isomorphic.
\end{proof}

The following description of $\WPic_{\Z\langle G\rangle}$ is 
reminiscent of the definition of class groups in algebraic number theory.

\begin{prop}
\label{Wpicasquot}
Let $\I_{\Z\langle G\rangle}$ denote the group of
invertible fractional $\Z\langle G\rangle$-ideals.
Then the group $\WPic_{\Z\langle G\rangle}$ is isomorphic to the quotient
of the group
$$
\{ (I,w) \in \I_{\Z\langle G\rangle} \times \Q\langle G\rangle^\ast : 
I\overline{I} = \Z\langle G\rangle w \text{ and } \psi(w)\in\R_{>0}
\text{ for all $\psi\in\Psi$} \} 
$$
by its subgroup 
$\{ (\Z\langle G\rangle v,v\overline{v}) : v\in \Q\langle G\rangle^\ast \}$.
\end{prop}

\begin{proof}
Define the map by $(I,w) \mapsto [L_{(I,w)}]$. 
Surjectivity follows from Theorem \ref{invertequiv},
and the kernel is the desired subgroup by Theorem \ref{I1I2isom}.
\end{proof}

Just as for the class group, we have:

\begin{thm}
\label{WPgpfin}
The group $\WPic_{\Z\langle G\rangle}$
is finite.
\end{thm}

\begin{proof}
If $L$ is an invertible $G$-lattice and $\{ b_1,\ldots, b_n\}$ is an
LLL-reduced basis, and for $\sigma \in G$ we have
$\sigma(b_i) = \sum_{j=1}^n a_{ij}^{(\sigma)}b_j$ 
with $a_{ij}^{(\sigma)}\in \Z$,
then 
$$
|\langle b_i,b_j\rangle| \le 2^{n-1} \quad \text{ and } \quad
|a_{ij}^{(\sigma)}| \le 3^{n-1}
$$ 
for all $i$, $j$, and $\sigma$,
by Proposition \ref{LLLlem}(iii) and (iv).
Thus there are only finitely many possibilities for 
$$
((\langle b_i,b_j\rangle)_{i,j=1}^n,(a_{ij}^{(\sigma)})_{i,j=1,\ldots,n; \sigma\in G}).
$$
If $L'$ is also an invertible $G$-lattice with
LLL-reduced basis $\{ b_1',\ldots, b_n'\}$, 
and if we have 
$$
\langle b_i,b_j\rangle =  \langle b_i',b_j'\rangle \quad \text{ and } \quad
a_{ij}^{(\sigma)}=a_{ij}'^{(\sigma)}
$$ 
for all $i$, $j$, and $\sigma$,
then the group isomorphism 
$$
L\to L', \quad b_i\mapsto b_i'
$$ 
is an
isomorphism of $G$-lattices.
The finiteness of $\WPic_{\Z\langle G\rangle}$ now follows. 
\end{proof}

We call $\WPic_{\Z\langle G\rangle}$ the {\bf Witt-Picard group} 
of $\Z\langle G\rangle$.
The reason for the nomenclature lies in Theorem \ref{invertequiv}.
If $R$ is a commutative ring, an invertible $R$-module is
an $R$-module $L$ for which
there exists an $R$-module $M$ with $L\otimes_R M \cong R$.
The Picard group $\Pic_R$ is the set 
of invertible $R$-modules up to isomorphism, where
the group operation is tensoring over $R$.
This addresses the module structure, while 
Witt rings reflect the structure as a unimodular lattice.

We remark that one can formulate algorithms for $\WPic_{\Z\langle G\rangle}$,
as follows.
Elements $[L] \in \WPic_{\Z\langle G\rangle}$ are
represented as $L$ with an LLL-reduced basis.

\begin{prop}
There are deterministic polynomial-time algorithms for:
\begin{enumerate}
\item
finding the unit element,
\item
inverting,
\item
multiplying,
\item
exponentiation,
\item
equality testing.
\end{enumerate}
\end{prop}

\begin{proof}
Part (i) is trivial, since $1 = [\Z\langle G\rangle]$.
For (ii) we have $[L]^{-1} = [\overline{L}]$, and the algorithm is
to replace each $\sigma$ by $\overline{\sigma}$.
For parts (iii), (iv), and (v) use Algorithms \ref{dd'alg} and  \ref{dtotheralg} below
and Theorem  \ref{mainthmcor}, respectively.
\end{proof}

\section{Multiplying and exponentiating invertible $G$-lattices}

In this section we give algorithms for 
multiplying and exponentiating invertible $G$-lattices.
We shall always assume that all $G$-lattices in inputs and
outputs of algorithms are specified via an LLL-reduced basis.
As we saw in the proof of Theorem \ref{WPgpfin},
this prevents coefficient blow-up.

\begin{algorithm}
\label{LMmultalg}
Given invertible $G$-lattices $L$ and $M$ equipped with LLL-reduced
bases, the algorithm outputs 
$L \otimes_{\Z\langle G\rangle} M$ with an LLL-reduced
basis and an $n\times n\times n$ array of integers to describe the
multiplication map 
$$
L \times M \to L \otimes_{\Z\langle G\rangle} M.
$$ 

\begin{enumerate}
\item
Realize $L$ as $L_{(I,w)}$ as in Lemma \ref{aimpliesdlem}, using
$\gamma = e_2$, and likewise realize $M$ as $L_{(J,v)}$.
\item
Compute $IJ \subset \Q\langle G\rangle$ and an LLL-reduced basis
for the $G$-lattice $L_{(IJ,wv)}$.
\item
Output $L \otimes_{\Z\langle G\rangle} M = L_{(IJ,wv)}$ and
the multiplication map  
$$
L \times M \to L \otimes_{\Z\langle G\rangle} M
$$
coming from multiplication $I\times J \to IJ$ in the ring 
$\Q\langle G\rangle$.
\end{enumerate}
\end{algorithm}

An alternative (probably less efficient) option is to directly use the definition
of tensor product, i.e., compute $L \otimes_{\Z\langle G\rangle} M$ 
as 
$$
(L \otimes_{\Z} M)/(\sum_{i,j,\sigma} \Z(\sigma b_i\otimes b_j' - b_i\otimes \sigma b_j'))
$$
where 
$$
L \otimes_{\Z} M = \bigoplus_{i,j} \Z(b_i\otimes b_j').
$$
With either choice, Algorithm \ref{LMmultalg} runs in polynomial time.
Using ideals works well for computing products and low powers
(cf.~Algorithm \ref{algor}(vii) below).
However, computing high powers of ideals cannot be done in polynomial time,
but computing high tensor powers of $G$-lattices is possible.
Likewise, the map 
$L \to L^{\otimes r}$,  $d \mapsto d^{\otimes r}$ 
cannot be written
down for large $r$, but one can compute the composition
$$
L \to L^{\otimes r} \to L^{\otimes r}/mL^{\otimes r}
$$
(see Algorithm \ref{dd'alg}), and thanks to Proposition \ref{findingekmprop2}
this suffices for our purposes.

Applying Algorithm \ref {LMmultalg} gives the following
polynomial-time algorithm.

\begin{algorithm}
\label{dd'alg}
Given $G$ and $u$ as usual, invertible $G$-lattices $L$ and $L'$ equipped with
LLL-reduced bases, a positive integer $m$,
and elements $d\in L/mL$ and $d'\in L'/mL'$, the algorithm computes 
$
L \otimes_{\Z\langle G\rangle} L'
$ 
and the element
$$
d\otimes d'\in (L \otimes L')/m(L \otimes L').
$$ 

\begin{enumerate}
\item 
Apply Algorithm \ref {LMmultalg} to compute $L \otimes_{\Z\langle G\rangle} L'$.
\item
Lift $d$ to $L$ and $d'$ to $L'$, and then apply the
composition
$$
L \times L' \to L \otimes_{\Z\langle G\rangle} L' \to 
(L \otimes L')/m(L \otimes L').
$$
\end{enumerate}
\end{algorithm}

For all $G$, $u$, and $m\in\Z_{>0}$, 
by the proof of Theorem \ref{WPgpfin} there is a bound on the runtime
of the previous algorithm that holds uniformly for all $L$, $L'$,
$d$, and $d'$, and this bound is polynomial in the length of the data
specifying $G$, $u$, and $m$.

Applying basis reduction, and iterating Algorithm \ref {dd'alg} using
an addition chain for $r$, gives the following
polynomial-time algorithm.
It replaces the polynomial chains in \S 7.4 of
the Gentry-Szydlo paper \cite{GS}.

\begin{algorithm}
\label{dtotheralg}
Given $G$, $u$, an invertible $G$-lattice $L$, positive integers $m$ and $r$, and
$d\in L/mL$,    
the algorithm computes 
$L^{\otimes r}$ and $d^{\otimes r}\in L^{\otimes r}/mL^{\otimes r}$. 
\end{algorithm}

Note that it is $\log(r)$ and not $r$ that enters in the runtime.
This means that very high powers of lattices can be computed without
coefficient blow-up, thanks to the basis reduction that takes place in
Algorithm \ref{LMmultalg}(ii). The fact that this is possible
was one of the crucial ideas of Gentry and Szydlo.

\section{The extended tensor algebra $\Lambda$}
\label{tensorsect}
The extended tensor algebra $\Lambda$ is a single algebraic structure
that comprises all rings and lattices that our main algorithm needs,
including their inner products.

Suppose $L$ is an invertible $G$-lattice. 
Letting $L^{\otimes 0} = \Z\langle G\rangle$ and 
letting
$$L^{\otimes m} =  
{L} \otimes_{\Z\langle G\rangle} \cdots \otimes_{\Z\langle G\rangle} {L} 
\quad \text{(with $m$ $L$'s)}$$  
and 
$$L^{\otimes (-m)} = \overline{L}^{\otimes m}= 
\overline{L} \otimes_{\Z\langle G\rangle} \cdots 
\otimes_{\Z\langle G\rangle} \overline{L}$$
for all $m \in \Z_{>0}$, 
define the extended tensor algebra 
$$
\Lambda  =  \bigoplus_{i\in\Z} L^{\otimes i} 
=  \ldots
\oplus \overline{L}^{\otimes 3} \oplus \overline{L}^{\otimes 2} \oplus \overline{L} \oplus
\Z\langle G\rangle \oplus L \oplus L^{\otimes 2} \oplus L^{\otimes 3} \oplus \ldots 
$$
(``extended'' because we extend the usual notion to include negative
exponents $L^{\otimes (-m)}$).
Each $L^{\otimes i}$ is an invertible $G$-lattice, and represents $[L]^i$.
For simplicity, we denote $L^{\otimes i}$ by $L^i$.
For all $j\in\Z$ we have 
$\overline{L^j} = \overline{L}^j = L^{-j}.$
Note that computing the $G$-lattice $L^{-1}=\overline{L}$ is trivial; just compose
the $G$-action map $G \to \GL(n,\Z)$ with the map
$G \to G$, $\sigma\mapsto \overline{\sigma}$.
The ring structure on $\Lambda$ is defined as the ring structure on the tensor algebra, supplemented with the lifted inner product $\cdot$ of Definition \ref{dotproddef}.
Let 
$\Lambda_\Q = \Lambda \otimes_\Z \Q.$

\begin{prop}
\label{LambdaProp}
\begin{enumerate}[leftmargin=*]
\item 
The extended tensor algebra 
$\Lambda$ is a commutative ring containing $\Z\langle G\rangle$ as a subring;
\item 
for all  $j\in \Z$, the action of $G$ on $L^j$ becomes multiplication in $\Lambda$;
\item 
$\Lambda$ has an involution $x\mapsto \overline{x}$ extending both the involution
of $\Z\langle G\rangle$ and the map $L \isom \overline{L}$;
\item 
if $j\in \Z$, then the lifted inner product 
$\cdot : L^j \times \overline{L^j} \to \Z\langle G\rangle$ 
becomes
multiplication in $\Lambda$, with $\overline{L^j}=\overline{L}^j$;
\item
if $j\in \Z$, then for all $x,y\in L^j$ we have 
$\langle x,y\rangle = t(x\overline{y});$
\item 
if $j\in \Z$ and $e \in L^j$ is short, then $\overline{e} = e^{-1}$ in $L^{-j}$;
\item
if $\gamma$ is as in Lemma \ref{aimpliesdlem}, then $\gamma\in \Lambda_\Q^\ast$,
one has 
$L_\Q^i = \Q\langle G\rangle\gamma^i$ 
for all $i\in\Z$,
and $\Lambda_\Q$ may be identified with the Laurent polynomial ring
$\Q\langle G\rangle [\gamma,\gamma^{-1}]$.
\item
if $e \in L$ is short, then 
$\Lambda = \Z\langle G\rangle [e,e^{-1}],$
where the right side is the subring of $\Lambda$ generated by
$\Z\langle G\rangle$, $e$, and $e^{-1}$, which is a Laurent polynomial ring.
\end{enumerate}
\end{prop}

\begin{proof}
The proof is straightforward. It is best to begin with (vii).
\end{proof}

All computations in $\Lambda$ and in 
$\Lambda/m\Lambda=\bigoplus_{i\in\Z} L^i/mL^i$  
with  $m\in\Z_{>0}$ that occur in our algorithms 
are done with homogeneous elements only, where 
the set of  homogeneous elements of $\Lambda$ is $\bigcup_{i\in\Z}L^i$.

If $A$ is a commutative ring, 
let $\mu(A)$ denote the subgroup of $A^\ast$ consisting of the roots of unity,
i.e., the elements of finite order.
The following result will allow us to construct a polynomial-time
algorithm to find $k$-th roots of short vectors, when they exist.

\begin{prop}
\label{rootsofunitythm}
Suppose $L$ is an invertible $G$-lattice, $\r\in\Z_{>0}$, and $\nu$ is a short vector 
in the $G$-lattice $L^r$. 
Let $$A=\Lambda/(\nu-1).$$
Identifying $\bigoplus_{i=0}^{\r-1} L^i \subset \Lambda$ with its
image in $A$, we can view 
$A = \bigoplus_{i=0}^{\r-1} L^i$ 
as a $\Z/\r\Z$-graded ring.
Then:
\begin{enumerate}

\item
$G \subseteq \mu(A) \subseteq \bigcup_{i=0}^{\r-1} L^i$,
\item
$\{ e\in L : e\cdot \bar{e}=1\} = \mu(A) \cap L$, 
\item
$|\mu(A)|$ is divisible by $2n$ and divides $2n\r$, 
\item
the degree map $\mu(A) \to \Z/r\Z$ that takes
$e\in \mu(A)$ to $j$ such that $e\in L^j$ is surjective if and only if 
$\mu(A) \cap L \neq \emptyset$,
and
\item
there exists  $e\in L$ for which $e\cdot \bar{e}=1$
if and only if $\#\mu(A)=2n\r$.
\end{enumerate}
\end{prop}

\begin{proof}
Since the ideal $$(\overline{\nu}-1) = (\nu^{-1}-1) = (1-\nu) = (\nu-1),$$
the map $a\mapsto \overline{a}$ induces an involution on $A$.

Next we show that the natural map 
$$\bigoplus_{i=0}^{\r-1} L^i \to \Lambda/(\nu-1)=A$$ is bijective.
For surjectivity, by Proposition \ref{LambdaProp}(vi)   
we have $\nu L^j = L^{j+r}$ for all $j\in\Z$, and thus
$L^{j+r}$ and $L^j$  
have the same image under the natural map $\Lambda\to \Lambda/(\nu-1)=A$.
For injectivity, suppose 
$$
0 \neq a = \sum_{i=h}^j a_i\in\Lambda
$$ 
with $h\le j$,
with all $a_i\in L^i$, and with $a_h \neq 0$ and $a_j \neq 0$.
Then 
$$
(\nu -1)a = \sum_{i=h}^{j+r} b_i
$$ 
with $b_i\in L^i$ where $b_h = -a_h \neq 0$ and $b_{j+r} = \nu a_j \neq 0$,
and therefore 
$$
(\nu -1)a \notin \bigoplus_{i=0}^{r-1}L^i.
$$
Hence we have 
$$(\nu -1)\Lambda \cap \bigoplus_{i=0}^{r-1}L^i = \{ 0\}.$$
The injectivity now follows.

Recall that $\Psi$ is the set of $\C$-algebra homomorphisms from 
$\C\langle G\rangle$ to $\C$.
Letting  
$A_\Q = A \otimes_\Z \Q,$ 
we have
$$
A_\Q = \Lambda_\Q/(\nu -1)\Lambda_\Q \quad \text{and }
\quad \Lambda_\Q = \bigoplus_{i\in\Z} L_\Q^i.
$$
Since $L$ is invertible, by Lemma \ref{aimpliesdlem}
there exists $\gamma\in L_\Q$ such that 
$$
L_\Q = \Q\langle G\rangle\cdot \gamma
$$
with $z=\gamma\overline{\gamma} \in \Q\langle G\rangle^\ast$
and $\psi(z) \in \R_{>0}$ for all $\psi\in\Psi$.
By Proposition \ref{LambdaProp}(vii) we have  
$\gamma\in L_\Q^\ast$, and 
$$
L_\Q^j = \Q\langle G\rangle\cdot \gamma^j
$$ 
for all $j\in\Z$, and 
$$\Lambda_\Q = \bigoplus_{i\in\Z} L_\Q^i = \Q\langle G\rangle[\gamma,\gamma^{-1}].$$
Thus, there exists $\delta\in \Q\langle G\rangle^\ast$
such that $\nu = \delta \gamma^r$.
The set of ring homomorphisms from $A$ to $\C$ can be identified
with the set of ring homomorphisms from $A_\Q$ to $\C$, which is 
$$\{ \text{ring homomorphisms $\varphi : \Lambda_\Q \to \C : \varphi(\nu) = 1$} \}.$$
The latter set  
can be identified with 
$$\{ (\psi,\zeta) : \psi\in\Psi , \zeta \in \C^\ast, \psi(\delta)\zeta^r = 1 \}$$
via the map 
$$\varphi \mapsto (\varphi|_{\Q\langle G\rangle},\varphi(\gamma))$$
and its inverse 
$$
(\psi,\zeta) \mapsto (\sum_i a_i\gamma^i \mapsto \sum_i \psi(a_i)\zeta^i),
$$ 
and has size $nr = \dim_\Q(A_\Q)$.
Since 
$$
1=\nu\overline{\nu} = (\delta \gamma^r) \overline{(\delta \gamma^r)} = 
\delta\overline{\delta} z^r,
$$ 
we have
$$\psi(\delta)\overline{\psi(\delta)}\psi(z)^r = 1 =
\psi(\delta)\overline{\psi(\delta)}(\zeta\overline{\zeta})^r,$$
so $\psi(z)^r = (\zeta\overline{\zeta})^r$.
Since $\psi(z) \in \R_{>0}$, we have $\psi(z) = \zeta\overline{\zeta}$.
Since $\overline{\gamma} = z\gamma^{-1}$, we now have 
$$
\varphi(\overline{\gamma}) = \varphi(z)\zeta^{-1}=\overline{\zeta} = \overline{\varphi(\gamma)}.
$$
By Lemma \ref{psimisclem}(i) we have 
$\psi(\bar{\alpha}) = \overline{\psi(\alpha)}$
for all $\alpha\in \Q\langle G\rangle$.
Since $A_\Q$ is generated as a ring by $\Q\langle G\rangle$ and
$\gamma$, it follows that $\varphi(\bar{\alpha}) = \overline{\varphi(\alpha)}$
for all $\alpha\in A_\Q$ and  all ring homomorphisms $\varphi : A_\Q \to \C$.

Applying Lemma \ref{KalglemBourb}(ii) to the commutative
$\Q$-algebra $A_\Q$ shows that 
$$
\bigcap_\varphi \ker \varphi = 0.
$$

Let 
$$
E = \{ e\in A: e\overline{e}=1 \},
$$ a subgroup of $A^\ast$.

If $e\in\mu(A)$, 
then  $\varphi(e)$ is a root of unity in $\C$
for all ring homomorphisms $\varphi:A \to \C$, so 
$$1 = \varphi(e)\overline{\varphi(e)} = 
\varphi(e)\varphi(\overline{e}) = 
\varphi(e\overline{e}).$$ Since $\bigcap_\varphi \ker \varphi = 0$, we have
$e\overline{e} =1$. 
Thus, $\mu(A) \subseteq E$.

Conversely, suppose $e\in E$. 
Write 
$e = \sum_{i=0}^{r-1} \varepsilon_i$ 
with $\varepsilon_i\in L^i$,
so $\overline{e} = \sum_{i=0}^{r-1} \overline{\varepsilon}_i$
with $\overline{\varepsilon}_i\in L^{-i} = L^{r-i}$ in $A$.
We have 
$$
1=e\overline{e} = 
\sum_{i=0}^{r-1} \varepsilon_i\overline{\varepsilon}_i,
$$
the degree $0$ piece of $e\overline{e}$.
Applying the map $t$ of Definition \ref{tdef} and using \eqref{tdoteqn}
we have $1 = \sum_{i=0}^{r-1} \langle \varepsilon_i, \varepsilon_i\rangle$.
It follows that there exists $j$ such that 
$\langle \varepsilon_j, \varepsilon_j \rangle = 1$, and 
$\varepsilon_i = 0$ if $i\neq j$.
Thus, 
$$E  \subseteq \bigcup_{i=0}^{r-1} \{ e\in L^i: \langle e, e\rangle=1 \},$$
giving (i). 
By Proposition \ref{eshortlem}(iii) and Example \ref{shortex} we have 
$E \cap \Z\langle\G\rangle = G,$ so
$\mu(\Z\langle\G\rangle) = G$.

The degree map from $E$ to $\Z/r\Z$ that takes
$e\in E$ to $j$ such that $e\in L^j$ is a group homomorphism
with kernel $E \cap \Z\langle\G\rangle = G$.
Therefore, $\#E$ divides $\#G \#(\Z/r\Z) = 2nr$.
Thus, $E \subseteq \mu(A) \subseteq E$, so $E = \mu(A)$ and we have
(ii) and (iii).
The degree map is surjective if and only if $\#\mu(A)=2n\r$,
and if and only if $1$ is in the image, i.e., if and only if
$\mu(A) \cap L \neq \emptyset$.
This gives (iv).
Part (v) now follows from (ii).
\end{proof}

\begin{rem}
In the proof of Proposition \ref{rootsofunitythm} we showed that
$\mu(\Z\langle\G\rangle) = G$.
\end{rem}

\section{Short vectors}
\label{shortvectalgsect}
Recall that $G$ is a finite abelian group of order $2n$ equipped with
an element $u$ of order $2$. 
The main result of this section is Algorithm \ref{findingealg}.

\begin{defn}
The exponent of a finite group $H$
is the least positive integer $k$ such that 
$\sigma^k = 1$ for all $\sigma\in H$. 
\end{defn}

The exponent of a finite group $H$ divides $\#H$ and has the same prime factors as $\#H$.

\begin{notation}
Let $k$ denote the exponent of $G$.
\end{notation}

By Theorem \ref{shortthm}, the $G$-isomorphisms $\Z\langle G\rangle \isom L$
for a $G$-lattice $L$ are in one-to-one correspondence with the short vectors of $L$, 
and if a short $e\in L$ exists, then the
short vectors of $L$ are exactly the $2n$ vectors
$\{ \sigma e : \sigma\in G\}$. 
With $k$ the exponent of $G$, we have  
$$(\sigma e)^k = \sigma^ke^k = e^k$$ in $\Lambda$. 
Hence for invertible $L$, all short vectors in $L$ have the same $k$-th power 
$e^k \in \Lambda$.
At least philosophically, it is easier to find things that are uniquely
determined. We look for $e^k$ first, and then recover $e$ from it.

The $n$ of \cite{GS} is an odd prime, so the group exponent $k=2n$, and 
$\Z\langle G\rangle$ embeds in
$\Q(\zeta_n)\times\Q$, where $\zeta_n\in\C^\ast$ is a primitive
$n$-th root of unity.
Since the latter is a product of only two number fields, the number of
zeros of $X^{2n}-v^{2n}$ is at most $(2n)^2$, and  
the Gentry-Szydlo method for finding $v$ from $v^{2n}$
is sufficiently efficient.
If one wants to generalize \cite{GS} to the case where $n$ is not prime,
then the smallest $t$ such that $\Z\langle G\rangle$ embeds in
$F_1\times\ldots\times F_t$ with number fields $F_i$ can be as large as $n$.
Given $\nu$, the number of zeros of $X^k-\nu$ could be as large as $k^t$.
Finding $e$ such that $\nu=e^k$ then requires a more efficient algorithm,
which we attain with Algorithm \ref{findingealg} below.

An {\bf order} is a commutative ring $A$ whose additive group
is isomorphic to $\Z^n$ for some $n\in \Z_{\ge 0}$. 
We specify an order by saying how to multiply any two vectors in a given basis.
In \cite{rootsofunity} we prove the following result,
and give the associated algorithm.

\begin{prop}
\label{rootsofunitylem}
There is a deterministic polynomial-time algorithm that, given an order $A$,
determines a set of generators for the group $\mu(A)$ of roots of unity in
$A^\ast$.
\end{prop}

\begin{algorithm}
\label{findingealg}
Given $G$ of exponent $k$, $u$,  
a fractional $\Z\langle G\rangle$-ideal $I$,
an element $w\in \Q\langle G\rangle^\ast$ such that 
$I\overline{I} = \Z\langle G\rangle \cdot w$ and
$\psi(w) \in\R_{>0}$ for all 
$\psi \in\Psi$, 
a short vector $\nu$ 
in the $G$-lattice $L_{(I^k,w^k)}$,
and the order 
$A = 
\bigoplus_{i=0}^{k-1} I^i$
with multiplication 
$$
I^i\times I^j \to I^{i+j}, \quad (x,y)\mapsto xy \quad
\text{ if \quad $i+j<k$}
$$ 
and 
$$
I^i\times I^j \to I^{i+j-k}, \quad (x,y)\mapsto xy/\nu \quad
\text{ if \quad $i+j\ge k$},
$$
the algorithm determines whether there exists $\alpha \in L_{(I,w)}$
such that $\nu=\alpha^k$ in $L_{(I^k,w^k)}$
and $\alpha\cdot \overline{\alpha}=1$, and if so,
finds one.

\begin{enumerate}
\item 
Apply Proposition \ref{rootsofunitylem} to compute generators for
$\mu(A)$. 
\item 
Apply the degree map  $\mu(A) \to \Z/k\Z$ from Proposition \ref{rootsofunitythm}(iv)
to the generators, and check whether
the images generate $\Z/k\Z$. If they do not,
output ``no $e$ exists'';
if they do, compute an element $\alpha\in\mu(A)$ whose
image under the degree map is $1$.  
\item
Check whether $\nu=\alpha^k$.
 If not, output ``no $\alpha$ exists''.
If so, output $\alpha$.
\end{enumerate}
\end{algorithm}

\begin{prop}
\label{findingethm}
Algorithm \ref{findingealg} produces correct output and runs in polynomial time.
\end{prop}

\begin{proof}
We apply Proposition \ref{rootsofunitythm} with $r=k$.
With $L=L_{(I,w)}$, our order $A$ can be identified with
the ring $\Lambda/(\nu-1)$ of that proposition. 
Suppose Step (ii) produces $\alpha\in\mu(A)$ of degree $1$.
Then $$\alpha\in\mu(A)\cap L_{(I,w)}= \{ \varepsilon\in L_{(I,w)} : \varepsilon\cdot \bar{\varepsilon}=1\}$$
by Proposition \ref{rootsofunitythm}(ii).
By Proposition \ref{eshortlem}(iii), this set is the set of
short vectors in $L_{(I,w)}$.
By Theorem \ref{shortthm}(iv), if a short $\varepsilon\in L_{(I,w)}$ exists, then the
short vectors in $L_{(I,w)}$ are exactly the $2n$ vectors
$\{ \sigma \varepsilon : \sigma\in G\}$, which all have the same $k$-th power 
since $k$ is the exponent of $G$.
By this and Proposition \ref{rootsofunitythm}(iv),
if any step fails then the desired $\alpha$ does not exist.
The algorithm runs in polynomial time since 
$$
\#\mu(A)= 2nk \le (2n)^2
$$
by Proposition \ref{rootsofunitythm}(v).
\end{proof}

\section{Finding auxiliary prime powers}
\label{LinnikApp}

In this section we present an algorithm to find auxiliary
prime powers $\ell$ and $m$. To 
bound the runtime, we 
use Heath-Brown's version of Linnik's theorem
in analytic number theory.

Recall that $G$ is a finite abelian group equipped with an element $u$ of order $2$, and
$k$ is the exponent of $G$.

\begin{notation}
For $m \in \Z_{>0}$ let 
$k(m)$ denote the exponent of the unit group $(\Z\langle G\rangle/(m))^\ast$.
\end{notation}

\begin{lem}
\label{LinnikLem}
Suppose $p$ is a prime number and 
$j\in\Z_{>0}$.
Then:
\begin{enumerate}
\item
$(\Z/p^j\Z)^\ast \subset (\Z\langle\G\rangle/(p^j))^\ast$; 
\item
if $p$ is odd, then the exponent of $(\Z/p^j\Z)^\ast$ is $(p-1)p^{j-1}$;
\item 
if $p\equiv 1$ {\rm mod} $k$, then $k(p^j) = (p-1)p^{j-1}$.
\end{enumerate}
\end{lem}

\begin{proof}
Parts (i) and (ii) are easy. For (iii),
we proceed by induction on $j$.
If $p\equiv 1$ mod $k$, then $p$ is odd.
We first take $j=1$. The map $x\mapsto x^p$ is a ring endomorphism
of $\Z\langle\G\rangle/(p)$ and is the identity on $G$, since the exponent
$k$ divides $p-1$.
Since $G$ generates the ring, the map is the identity and
therefore $x^p=x$ for all
$x\in \Z\langle\G\rangle/(p)$ and $x^{p-1}=1$ for all
$x\in (\Z\langle\G\rangle/(p))^\ast$.

Now suppose $j>1$. Suppose $x\in \Z\langle\G\rangle$ maps to a unit
in $\Z\langle\G\rangle/(p^{j})$.
By the induction hypothesis, 
$$
x^{(p-1)p^{j-2}} \equiv 1 \mod p^{j-1}.
$$
Thus, $x^{(p-1)p^{j-2}} = 1 + p^{j-1}v$ for some $v\in\Z\langle\G\rangle$.
Since $(j-1)p \ge j$ we have
$$
x^{(p-1)p^{j-1}} = (1 + p^{j-1}v)^p = 
1 + \binom{p}{1}p^{j-1}v + \cdots + p^{(j-1)p}v^p \equiv 1 \mod p^j.
$$
Thus, $k(p^j)$ divides $(p-1)p^{j-1}$ for all $j\in\Z_{>0}$.
Part (iii) now follows from (i) and (ii).
\end{proof}

\begin{thm}[Heath-Brown, Theorem 6 of \cite{HB}]
\label{LinnikHeathBrown}
There is an effective constant $c>0$ such that if
$a,t \in \Z_{>0}$ and $\gcd(a,t)=1$, then the smallest
prime $p$ such that $p \equiv a$ mod $t$ is at most $ct^{5.5}$.
\end{thm}

\begin{algorithm}
\label{LinnikThmAlg0}
Given positive integers $n$ and $k$ with $k$ even, the algorithm
produces prime powers $\ell = p^r$ and $m = q^s$ with 
$\ell, m \ge 2^{n/2} + 1$ such that
$p\equiv q\equiv 1$ mod $k$ and  
$\gcd(\varphi(\ell), \varphi(m)) = k$,
where $\varphi$ is Euler's phi function.

\begin{enumerate}
\item
Try $p=k+1, 2k+1, 3k+1,\ldots$ until the least prime 
$p\equiv 1$ mod $k$ is found.
\item 
Find the smallest $r\in\Z_{> 0}$ such that $p^r \ge 2^{n/2}+1$.
\item
Try $q=p+k, p+2k, \ldots$ until the least prime 
$q\equiv 1$ mod $k$ such that $\gcd((p-1)p,q-1) =k$ is found. 
\item
Find the smallest $s\in\Z_{> 0}$ such that $q^s \ge 2^{n/2}+1$.
\item
Let $\ell = p^r$ and $m=q^s$.
\end{enumerate}
\end{algorithm}

\begin{prop}
\label{LinnikThm}
Algorithm \ref{LinnikThmAlg0} runs in time $(n+k)^{O(1)}$.
\end{prop}

\begin{proof}
Algorithm \ref{LinnikThmAlg0} takes as input $n,k\in\Z_{>0}$ with $k$ even, 
and computes positive integers
$r$ and $s$ and primes $p$ and $q$ such that:
\begin{itemize}
\item
$p\equiv q\equiv 1$ mod $k$, 
\item
$\gcd((p-1)p^{r-1},(q-1)q^{s-1}) =k$, 
\item
$p^r \ge 2^{n/2}+1$, and
\item
$q^s \ge 2^{n/2}+1$.
\end{itemize}

We next show that Algorithm \ref{LinnikThmAlg0} terminates, with correct output, 
in the claimed time.
By Theorem \ref{LinnikHeathBrown} above,
the prime $p$ found by Algorithm \ref{LinnikThmAlg0} satisfies
$p \le ck^{5.5}$ with an effective constant $c>0$.
Primality testing can be done by trial division.
If $p-1=k_1k_2$ with every prime divisor of $k_1$ also dividing $k$ and
with $\gcd(k_2,k)=1$, then to have 
$$
\gcd((p-1)p,q-1) =k
$$ 
it suffices 
to have 
$$q\equiv 2 \text{ mod $p$ \,\,  and  \,\, $q\equiv 1+k$ mod $k_1$  \,\,  and  \,\,  
$q\equiv 2$ mod $k_2$.}
$$ 
This gives a congruence 
$$
q\equiv a \text{ mod $p(p-1)$}
$$ 
for some $a$ with $\gcd(a,p(p-1))=1$.
Theorem \ref{LinnikHeathBrown} implies that Algorithm \ref{LinnikThmAlg0} 
produces a prime $q$ with the desired properties and satisfying 
$$q \le c(p^2)^{5.5} 
\le c(ck^{5.5})^{11}= 
c^{12}k^{60.5}.$$ 
The upper bounds on $p$ and $q$ imply that Algorithm \ref{LinnikThmAlg0} runs in
time $(n+k)^{O(1)}.$
\end{proof}

\begin{rem}
In practice, Algorithm~\ref{LinnikThmAlg0} is {\em much} faster than
implied by the proof of Proposition~\ref{LinnikThm};
Theorem \ref{LinnikHeathBrown} is unnecessarily pessimistic,
and in practice one does not need to find a prime $q$ that is congruent to
$2$~mod~$pk_2$ and to $1+k$~mod~$k_1$.
In work in progress, we get better bounds for the runtime of
our main algorithm, and avoid using the
theorem of Heath-Brown or Algorithm \ref{LinnikThmAlg0}, 
by generalizing our theory to the setting
of ``CM orders''.
\end{rem}

Algorithm \ref{LinnikThmAlg0}  
immediately yields the following algorithm.

\begin{algorithm}
\label{LinnikThmAlg}
Given $G$ and $u$, the algorithm
produces prime powers $\ell$ and $m$ such that
$$
\ell, m \ge 2^{n/2} + 1 \quad \text{ and \quad $\gcd(k(\ell), k(m)) = k$,}
$$
where $k$ is the exponent of $G$, and produces the
values of $k(\ell)$ and $k(m)$.

\begin{enumerate}
\item
Compute $n$ and $k$.
\item 
Run Algorithm \ref{LinnikThmAlg0} to compute prime powers 
$\ell=p^r$ and $m=q^s$  with
$$
\ell, m \ge 2^{n/2} + 1
$$ 
such that
$$
p\equiv q\equiv 1 \text{ mod $k$  
 \quad and \quad $\gcd(\varphi(\ell), \varphi(m)) = k$.}
$$
\item
Compute $k(\ell) = (p-1)p^{r-1}$ and $k(m) = (q-1)q^{s-1}$.
\end{enumerate}
\end{algorithm}

By Lemma \ref{LinnikLem}(iii),
Algorithm \ref{LinnikThmAlg} produces the desired output.
It follows from Proposition \ref{LinnikThm} that 
Algorithm~\ref{LinnikThmAlg} runs in 
polynomial time (note that the input 
in Algorithm~\ref{LinnikThmAlg} includes the group 
law on $G$).

\begin{rem}
Our prime powers $\ell$ and $m$ play the roles that in the
Gentry-Szydlo paper \cite{GS} 
were played by auxiliary prime numbers 
$$
P, P' > 2^{(n+1)/2}
$$
such that 
$$
\gcd(P-1,P'-1)=2n.
$$ 
Our $k(\ell)$ and $k(m)$ replace their $P-1$ and $P'-1$.
While the Gentry-Szydlo primes $P$ and $P'$ are found with at best a
probabilistic algorithm, we can find $\ell$ and $m$ in
polynomial time with a deterministic algorithm.
(Further, the ring elements they work with were required to not be zero divisors modulo
$P$, $P'$ and other small auxiliary primes;
we require no analogous condition on $\ell$ and $m$,
since by Definition \ref{invertdef2}, when $L$ is invertible then for
{\em all} $m$, the $(\Z/m\Z)\langle G\rangle$-module $L/mL$ is free of
rank $1$.) 
\end{rem}

The next result will provide the proof of correctness for a key step in our main algorithm.

\begin{lem}
\label{lemmaForAlgorithm}
Suppose  
$e$ is a short vector in an invertible $G$-lattice $L$, 
suppose $\ell, m\in\Z_{\ge 3}$, and
suppose $e_{\ell m}\in L$ is such that $e_{\ell m} + \ell mL$ generates $L/\ell mL$
as a $(\Z/\ell m\Z)\langle G\rangle$-module. 
Then $e^{k(m)}$ is the unique short vector in 
the coset $$e_{\ell m}^{k(m)} + mL^{k(m)},$$ and there is a unique
$s \in ((\Z/\ell\Z)\langle G\rangle)^\ast$ such that
$$
e^{k(m)} \equiv se_{\ell m}^{k(m)} \mod \ell L^{k(m)}.
$$
If further $b\in\Z_{>0}$ and $bk(m)\equiv k$ mod $k(\ell)$,
then $e^k$ is the unique short vector in
$s^be_{\ell m}^k + \ell L^{k}$.
\end{lem}

\begin{proof}
Since $e$ is short, we have $\Z\langle\G\rangle e = L$. Thus
for all $r\in\Z_{>0}$, the coset
$e+rL$ generates $L/rL$ as a $\Z\langle\G\rangle/(r)$-module.
We also have that $e_{\ell m} + mL$ 
generates $L/mL$ as a $\Z\langle\G\rangle/(m)$-module,
and $e_{\ell m} + \ell L$ generates $L/\ell L$ as a 
$\Z\langle\G\rangle/(\ell)$-module.
Thus, there exist $y_m\in (\Z\langle\G\rangle/(m))^\ast$
and $y_\ell\in (\Z\langle\G\rangle/(\ell))^\ast$ such that
$$
e_{\ell m}=y_me \text{ mod $mL$ \quad and \quad $e_{\ell m}=y_\ell e$ mod $\ell L$.}
$$
It follows that
$$
e_{\ell m}^{k(m)}\equiv e^{k(m)}\text{ mod ${m}L^{k(m)}$ \quad
and \quad $e_{\ell m}^{k(\ell)}\equiv e^{k(\ell)}$ mod ${\ell}L^{k(\ell)}$.}
$$

We have 
$$
(\Z/\ell\Z)\langle G\rangle e = L/\ell L = 
(\Z/\ell\Z)\langle G\rangle e_{\ell m}.
$$
Thus $$(\Z/\ell\Z)\langle G\rangle\cdot e^{k(m)} = L^{k(m)}/\ell L^{k(m)} = 
(\Z/\ell\Z)\langle G\rangle\cdot e_{\ell m}^{k(m)},$$ so
\begin{equation}
\label{sdefneq}
e^{k(m)} \equiv se_{\ell m}^{k(m)} \mod \ell L^{k(m)}
\end{equation} 
for a unique $s \in ((\Z/\ell\Z)\langle G\rangle)^\ast$.
We have $e \cdot \bar{e} = 1$, so 
$$e\in\Lambda^\ast \quad \text{ and \quad
$e+\ell\Lambda \in (\Lambda/\ell\Lambda)^\ast$.}
$$
By \eqref{sdefneq} we have 
$$
(e+\ell\Lambda)^{k(m)} = s(e_{\ell m}+\ell\Lambda)^{k(m)}
$$ 
in $\Lambda/\ell\Lambda = \bigoplus_{i\in\Z} L^i/\ell L^i$.
It follows that 
$$
e_{\ell m}+\ell\Lambda \in (\Lambda/\ell\Lambda)^\ast.
$$

If $ak(\ell)+bk(m)=k$ with $a\in\Z$, then 
$$e^k = (e^{k(\ell)})^a(e^{k(m)})^b 
\equiv (e_{\ell m}^{k(\ell)})^a(se_{\ell m}^{k(m)})^b \equiv 
s^be_{\ell m}^k \mod \ell\Lambda,$$ so
$s^be_{\ell m}^k + \ell L^k$ contains the short vector $e^k$ of $L^k$.
In both cases, uniqueness follows from Proposition \ref{findingekmprop}.
\end{proof}

\section{The main algorithm}
\label{algorsect}

Algorithm \ref{algor} below is the algorithm promised in Theorem \ref{mainthm}.
That it is correct
and runs in polynomial time follows from the
results above; see the discussion after the algorithm.
As before, $k$ is the exponent of the group $G$ and  
$k(j)$ is the exponent of $(\Z\langle G\rangle/(j))^\ast$ if
$j\in\Z_{>0}$.

\begin{algorithm}
\label{algor}
Given $G$, $u$, and a $G$-lattice $L$, the algorithm determines whether
there exists a $G$-isomorphism $\Z\langle G\rangle \isom L$, 
and if so, computes one.
\begin{enumerate}

\item 
Apply Algorithm \ref{elemprop2} to check whether $L$ is invertible. 
If it is not, terminate with ``no''.
\item 
Apply Algorithm \ref{LinnikThmAlg} to produce prime powers $\ell$ and $m$
as well as $k(\ell)$ and $k(m)$.
\item 
Use Proposition \ref{elemprop} to compute $e_{\ell m}$ and $e_2$.
\item 
Use Algorithm \ref{dtotheralg}
to compute the pair 
$$
(L^{k(m)},e_{\ell m}^{k(m)} + \ell mL^{k(m)}).
$$
Use Algorithm \ref{findingekmalg2}
to decide whether the coset 
$$
e_{\ell m}^{k(m)} + mL^{k(m)}
$$
contains a short vector $\nu_m \in L^{k(m)}$, and if so, compute it. 
Terminate with ``no'' if none exists.
\item
Compute $s \in (\Z/\ell\Z)\langle G\rangle$
such that
$$
\nu_m = se_{\ell m}^{k(m)} + \ell L^{k(m)}
$$ 
in $L^{k(m)}/\ell L^{k(m)}$.
\item 
Use the extended Euclidean algorithm to find  
$b\in\Z_{>0}$ such that 
$$
bk(m)\equiv k \text{ mod $k(\ell)$.}
$$
\item
Compute 
$$
I = \{ x\in \Q\langle G\rangle : xe_{2} \in L\}
$$
and compute $I^i$ for $i=2,\ldots,k$.
\item
Compute $s^b \in (\Z/\ell\Z)\langle G\rangle$ 
and  
$$
s^b(e_{\ell m}/e_2)^k + \ell I^{k} \in I^{k}/\ell I^{k}.
$$
Use Algorithm \ref{findingekmalg2}
to decide whether the coset 
$$s^b(e_{\ell m}/e_2)^k + \ell I^{k}$$
contains a short vector $\nu$ for the lattice $L_{(I^k,w^k)}$
where $$w = (e_2\cdot \overline{e_2})^{-1},$$
and if so, compute it. 
Terminate with ``no'' if none exists.
\item 
Construct the order $A = \bigoplus_{i=0}^{k-1} I^i$ with 
multiplication 
$$
I^i\times I^j \to I^{i+j}, \quad (x,y)\mapsto xy
\quad \text{if $i+j<k$}
$$ 
and 
$$
I^i\times I^j \to I^{i+j-k}, \quad (x,y)\mapsto xy/\nu
\quad \text{if $i+j\ge k$.}
$$
Apply Algorithm \ref{findingealg} to find
$\alpha\in L_{(I,w)}$ such that $\nu=\alpha^k$ and 
$\alpha \cdot \overline{\alpha}=1$ (or to prove there is no $G$-isomorphism).
Let $e=\alpha e_2\in L$, and let the map $\Z\langle G\rangle \isom L$ send $x$ to $xe$.
\end{enumerate}
\end{algorithm}

\begin{prop}
\label{algorpf}
Algorithm \ref{algor} is a deterministic polynomial-time algorithm 
that, given a finite abelian group $G$, 
an element $u\in G$ of order $2$, and a $G$-lattice $L$,
outputs a $G$-isomorphism $\Z\langle G\rangle \isom L$ or a proof that none exists.
\end{prop}

\begin{proof}
By Theorem \ref{shortthm}(iii), the
$G$-lattice $L$ is $G$-isomorphic to $\Z\langle G\rangle$ 
if and only if $L$ is invertible and has a short vector.
Algorithm \ref{elemprop2} checks whether $L$ is invertible. 
If it is, 
we look for an $e\in L$ such that $e\bar{e}=1$.

Algorithm \ref{LinnikThmAlg} produces prime powers 
$\ell, m \ge 2^{n/2} + 1$ such that 
$$\gcd(k(\ell), k(m)) = k.$$
The algorithm in Proposition~\ref{elemprop} produces $e_{\ell m}$,  
which then serves as both $e_m$ and $e_\ell$.
Algorithm \ref{findingekmalg2}
finds a short vector $\nu_{m}$ (if it exists) in the coset 
$$e_{\ell m} + mL^{k(m)} \in L^{k(m)}/mL^{k(m)}.$$
If $e\in L$ is short, then $\nu_m = e^{k(m)}$ by Lemma \ref{lemmaForAlgorithm}.

As in Lemma \ref{aimpliesdlem},
the set $I$ is an invertible $\Z\langle G\rangle$-ideal,
and the map 
$$
L_{(I,w)} \isom L, \qquad x\mapsto xe_2
$$
is an isomorphism of $G$-lattices, so $L = Ie_2$.
We next show that $I^i$ for $i=2,\ldots,k$ can be computed in
polynomial time.
Let 
$$
q = (L: \Z\langle G\rangle e_2).
$$
Then $L = \Z\langle G\rangle e_2 + \Z\langle G\rangle e_q$,
so $I = \Z\langle G\rangle  + \Z\langle G\rangle \beta$
where $\beta \in \Q\langle G\rangle$ and
$\beta = e_q/e_2 \in  \Lambda_\Q$.
We claim that $$I^i = \Z\langle G\rangle + \Z\langle G\rangle \beta^i$$
for all $i\in\Z_{>0}$. Namely, we have 
$$
L \supset \Z\langle G\rangle e_2 \supset  qL,
$$ 
so
$L^i \supset \Z\langle G\rangle e_2^i \supset  q^iL^i$.
Since $L^i = I^ie_2^i$, we have
$$
I^i \supset \Z\langle G\rangle \supset q^iI^i.
$$
Similarly, letting $r = (L: \Z\langle G\rangle e_q)$ we have
$$
I^i \supset \Z\langle G\rangle \beta^i \supset r^iI^i.
$$
Since $q$ and $r$ are coprime by Lemma \ref{inviiilem}(i), we have 
$$
I^i \supset \Z\langle G\rangle + \Z\langle G\rangle \beta^i \supset 
q^iI^i + r^iI^i = I^i,
$$ 
and the desired equality follows.
Now $\beta, \beta^2,\ldots,\beta^k$ are easily computable in
polynomial time, since $k\le 2n$.
 
By Lemma \ref{lemmaForAlgorithm},
if $\alpha\in L_{(I^k,w^k)}$   
is short then 
$\nu=\alpha^k$. 
Algorithm \ref{findingealg} then finds a short vector 
$\alpha\in L_{(I^k,w^k)}$,  
or proves that none exists.
Then $e=\alpha e_2$ is a short vector in $L$, and
the map $x\mapsto xe$ gives the desired $G$-isomorphism from 
$\Z\langle G\rangle$ to $L$.
\end{proof}

\begin{rem}
There is a version of the algorithm in which checking invertibility in
step (i) is skipped. In this case, the algorithm may misbehave at
other points, indicating that $L$ is not invertible and thus not
$G$-isomorphic to $\Z\langle G\rangle$ by Lemma \ref{stdGex}.
At the end one would check whether 
$\langle e,  e\rangle =1$ and
$\langle e, \sigma e\rangle =0$ for all $\sigma \neq 1,u$.
If so, then 
$\{\sigma e\}_{\sigma\in S}$ is an orthonormal basis for $L$,
and $x\mapsto xe$ gives the desired isomorphism; if not,
no such isomorphism exists.
\end{rem}

Thanks to Corollary \ref{invertequivcor2}, we can convert Algorithm \ref{algor}
to an algorithm to test whether two $G$-lattices are $G$-isomorphic (and
produce an isomorphism).

\begin{algorithm}
\label{algorLM}
Given $G$, $u$, and two invertible $G$-lattices $L$ and $M$, the algorithm
determines whether there is a $G$-isomorphism $M \isom L$, 
and if so, computes one.
\begin{enumerate}

\item 
Compute $L \otimes_{\Z\langle G\rangle} \overline{M}$.
\item
Apply Algorithm \ref{algor} to find   
a $G$-isomorphism $$\Z\langle G\rangle \isom L \otimes_{\Z\langle G\rangle} \overline{M},$$ or a proof that none exists. In the latter case, terminate with ``no''.
\item
Using this map and the map 
$$\overline{M} \otimes_{\Z\langle G\rangle} M \to \Z\langle G\rangle, \quad
\overline{y} \otimes x \mapsto \overline{y} \cdot x,$$
output the composition of the (natural) maps
$$M \isom \Z\langle G\rangle \otimes_{\Z\langle G\rangle} M \isom
L \otimes_{\Z\langle G\rangle} \overline{M}\otimes_{\Z\langle G\rangle} M
\isom L \otimes_{\Z\langle G\rangle} \Z\langle G\rangle \isom L.$$
\end{enumerate}
\end{algorithm}

\bigskip

\noindent{\bf Acknowledgments:} 
We thank Craig Gentry, Daniele Micciancio, Ren\'e Schoof, Mike Szydlo,
and all the participants of the 2013 Workshop on Lattices with Symmetry.
We thank the reviewers for very helpful comments.

\bigskip

\noindent{\bf Note from the Editor in Chief:} 
This paper was solicited by the editors,
due to the extended abstract \cite{LenSil} on which it is based 
having been selected as one of the best papers at the conference Crypto 2014.

\end{document}